\documentclass[11pt]{preprint}
\usepackage{amsmath}
\usepackage{amsfonts}
\usepackage{amssymb}
\usepackage{amsthm,color}
\usepackage{newlfont}
\usepackage{bbm}
\usepackage[normalem]{ulem}
\usepackage{authblk}
\usepackage[toc,page]{appendix}

\newtheorem{Theorem}{Theorem}[section]
\newtheorem{Definition}[Theorem]{Definition}

\newtheorem{Lemma}[Theorem]{Lemma}

\newtheorem{Remark}[Theorem]{Remark}

\numberwithin{equation}{section}

\begin{document}

\def\rb{\color{red}}
\def\re{\color{black}}
\def\le{\left}
\def\r{\right}
\def\cost{\mbox{const}}
\def\a{\alpha}
\def\d{\delta}
\def\ph{\varphi}
\def\e{\varepsilon}
\def\la{\lambda}
\def\si{\sigma}
\def\La{\Lambda}
\def\B{{\cal B}}
\def\A{{\mathcal A}}
\def\L{{\mathcal L}}
\def\O{{\mathcal O}}
\def\bO{\overline{{\mathcal O}}}
\def\F{{\mathcal F}}
\def\K{{\mathcal K}}
\def\H{{\mathcal H}}
\def\D{{\mathcal D}}
\def\C{{\mathcal C}}
\def\M{{\mathcal M}}
\def\N{{\mathcal N}}
\def\G{{\mathcal G}}
\def\T{{\mathcal T}}
\def\R{{\mathbb R}}
\def\I{{\mathcal I}}

\def\bw{\overline{W}}
\def\phin{\|\varphi\|_{0}}
\def\s0t{\sup_{t \in [0,T]}}
\def\lt{\lim_{t\rightarrow 0}}
\def\iot{\int_{0}^{t}}
\def\ioi{\int_0^{+\infty}}
\def\ds{\displaystyle}

\def\pag{\vfill\eject}
\def\fine{\par\vfill\supereject\end}
\def\acapo{\hfill\break}

\def\beq{\begin{equation}}
\def\eeq{\end{equation}}
\def\barr{\begin{array}}
\def\earr{\end{array}}
\def\vs{\vspace{.1mm}   \\}
\def\rd{\reals\,^{d}}
\def\rn{\reals\,^{n}}
\def\rr{\reals\,^{r}}
\def\bD{\overline{{\mathcal D}}}
\newcommand{\dimo}{\hfill \break {\bf Proof - }}
\newcommand{\nat}{\mathbb N}
\newcommand{\E}{\mathrm E}
\newcommand{\Pro}{\mathbb P}
\newcommand{\com}{{\scriptstyle \circ}}
\newcommand{\reals}{\mathbb R}
\def\Tr{\textnormal{Tr}}
\def\P{\mathrm{P}}

\title{On dynamical systems perturbed by a null-recurrent motion: The general case}

\author[1,3]{Zsolt Pajor-Gyulai \thanks{zsolt@cims.nyu.edu}}
\author[2,3]{Michael Salins \thanks{msalins@bu.edu}}
\affil[1]{Courant Institute of Mathematical Sciences, New York University, New York, NY}
\affil[2]{Department of Mathematics and Statistics, Boston University, Boston, MA}
\affil[3]{Department of Mathematics, University of Maryland, College Park, MD}

\institute{}

\date{}

\maketitle

\begin{abstract}
  We consider a perturbed ordinary differential equation where the perturbation is only significant when a one-dimensional null recurrent diffusion is close to zero. We investigate the first order correction to the unperturbed system and prove a central limit theorem type result, i.e., that the normalized deviation process converges in law in the space of continuous functions to a limit process which we identify. We show that this limit process has a component which only moves when the limit of the null-recurrent fast motion equals zero. The set of these times forms a zero-measure Cantor set and therefore the limiting process cannot be described by a standard SDE. We characterize this process by its infinitesimal generator (with appropriate boundary conditions) and we also characterize the process as the weak solution of an SDE that depends on the local time of the fast motion process.  We also investigate the long time behavior of such a system when the unperturbed motion is trivial. In this case, we show that the long-time limit is constant on a set of full Lebesgue measure with probability 1, but it has nontrivial drift and diffusion components that move only when the fast motion equals zero.
\end{abstract}

\section{Introduction}
In systems of multidimensional diffusion processes where both slow and fast timescales are present, the limiting behavior of the slow component is an interesting mathematical problem with applications in physics, biology and other areas. One possible setting is a system of diffusion processes $(\tilde{X}^{\varepsilon}(t),\tilde{Y}^{\varepsilon}(t))$ that depends on a parameter $\varepsilon$ representing the ratio of the two time scales. More precisely, suppose that $\tilde{X}^{\varepsilon}$ changes faster and faster in time as $\varepsilon\downarrow 0$ while $\tilde{Y}^{\e}$ changes on the same time scale for all values of $\e$. Then $\tilde{X}^{\e}$ and $\tilde{Y}^{\e}$ are called the fast and the slow component respectively.

Classical stochastic averaging theory, first invented by Khasminskii (\cite{K68}) expanding the ideas of Bogolyubov for ordinary differential equations (\cite{B45}), states that if the fast component $X^{\varepsilon}$ has a finite invariant measure, then as $\varepsilon\to 0$, the distribution of the slow component $Y^{\e}$ approaches a limit where the $X$ dependence in the drift and diffusion coefficients of $Y$ is averaged out with respect to this invariant measure. This result was later extended and refined by a vast number of authors (see e.g \cite{FW12},\cite{GS82},\cite{K04},\cite{KY04},\cite{P77},\cite{S89} amongst many others).

In this paper, we are going to study a system where the fast component is a null-recurrent process and therefore has no invariant probability measure while the slow motion is a perturbation of an ordinary differential equation. Namely, let $W$ be a $k$-dimensional Brownian motion and consider the system

\begin{align}
\label{eq:X_eqfast} d\tilde{X}^{\e}(t)&=\e^{-1}\varphi(\tilde{X}^{\e}(t),\tilde{Y}^{\e}(t))dW(t),&\tilde{X}^{\e}(0)=x_0/\e,\\
\label{eq:Y_eqfast} d\tilde{Y}^{\varepsilon}(t)&=b(\tilde{X}^{\e}(t),\tilde{Y}^{\varepsilon}(t))dt+\sigma(\tilde{X}^{\e}(t),\tilde{Y}^{\varepsilon}(t))dW(t),&\tilde{Y}^{\e}(0)=y_0,
\end{align}
where all coefficients are Lipschitz continuous, $\varphi$ is bounded away from zero and infinity, and the process $\tilde{X}^{\e}$ ($\tilde{Y}^{\e}$) is one (d) dimensional. Intuitively, the fast component $\tilde{X}^{\e}(t)$ is of order $\mathcal{O}(\e^{-1})$ and therefore the limiting dynamics of the system, as $\e\downarrow 0$, is governed by the behavior of the coefficients for large absolute value of the $x$-variable.

 In \cite{KK04}, it was shown that if $b$, $\varphi$, and $\sigma$ have certain limits as $|x|\uparrow\infty$ in the Cesaro sense then $(\e\tilde{X}^{\e}(t),\tilde{Y}^{\e}(t))$ converges in law to a process where the $x$-dependence in (\ref{eq:X_eqfast})-(\ref{eq:Y_eqfast}) is replaced by coefficients associated to these limits.  More precisely, after the change of coordinates $(X^{\e}(t),Y^{\e}(t))=(\e\tilde{X}^{\e}(t),\tilde{Y}^{\e}(t))$, the system $(X^{\e}(t),Y^{\e}(t))$ can be written as
\[
d\left(\begin{array}{c}
X^{\e}(t)\\
Y^{\e}(t)
\end{array}\right)=\left(\begin{array}{c}
0\\
b(\e^{-1}X^{\e}(t),Y^{\e}(t))
\end{array}\right)dt+\left(\begin{array}{c}
\varphi(\e^{-1}X^{\e}(t),Y^{\e}(t))\\
\sigma(\e^{-1}X^{\e}(t),Y^{\e}(t))
\end{array}\right)dW(t)
\]
It might be tempting to assume that the limit will be given by taking the limit of the coefficients for large positive and negative values of $x$. However, this is not the complete picture as one has to include the effects of null recurrent averaging. To see this, note that the diffusion matrix can be written as
\[
A(x,y)=\frac{1}{2}\left(\begin{array}{cc}
|\varphi(x,y)|^2 &\varphi\sigma^T(x,y)\\
\sigma\varphi^T(x,y) & \sigma\sigma^T(x,y)
\end{array}\right)=\frac{1}{2}\frac{1}{1/|\varphi(x,y)|^2}\left(\begin{array}{cc}
 1&\frac{\varphi\sigma^T(x,y)}{|\varphi(x,y)|^2}\\
\frac{\sigma\varphi^T(x,y)}{|\varphi|^2(x,y)}& \frac{\sigma\sigma^T(x,y)}{|\varphi^2(x,y)|}
\end{array}\right)
\]
 Let $p(x,y)=1/|\varphi(x,y)|^2$ and note that if we freeze the $y$ dynamics, this is proportional to the sigma finite invariant density of the $X^{\e}$ process. For a function $K(x,y)$, define
\[
K_\pm(x,y) = \chi_{\{x>0\}} \cdot \lim_{u \to +\infty} \frac{1}{u}\int_0^u K(v,y) dv + \chi_{\{x<0\}} \cdot \lim_{u \to -\infty} \frac{1}{u} \int_0^u K(v,y) dv,
\]
whenever the Cesaro-limits exist. Using this, introduce the averaged diffusion matrix $\bar{A}_{ij}(x,y)=(pA_{ij})_{\pm}(x,y)/p_{\pm}(x,y)$ for $i,j=1,..,1+d$ and the averaged drift $\bar{b}(x,y)=(0, (pb)_{\pm}/p_{\pm})$.

\begin{Theorem}[Khasminskii, Krylov \cite{KK04}]
Suppose that $\varphi,\sigma, b$ are Lipschitz continuous and, for each $x$, their first and second order derivatives in $y$ are bounded continuous functions of $y$. Assume as well that for some positive constants $c_1, c_2, c_3>0$, we have
\[
c_1\leq |\varphi(x,y)|\leq c_2,\qquad \sum_{i=1}^{d}(A^T(x,y))_{ii}+b_i^2(x,y))\leq c_3(1+|y|^2)
\]
for all $(x,y)\in\mathbb{R}^{1+d}$. Also assume that $p, pA_{ij}, pb_i$ and all their first and second derivatives in $y$ have well defined Cesaro limits as $x\rightarrow \pm \infty$. Then the pair $(\e \tilde{X}^{\e}(t), \tilde{Y}^{\e}(t))$ converges in distribution in $C\left([0,\infty),\mathbb{R}^{1+d}\right)$ to the solution of
\[
d\left(X^{0}(t),Y^{0}(t)\right)=\bar{b}(X^0(t),Y^{0}(t))dt+\sqrt{\bar{A}(X^0(t),Y^{0}(t))}dW(t),
\]
whenever weak uniqueness for this equation holds.
\end{Theorem}

 In general, however, the Ces\'aro limits in question can be quite different for large positive and negative values of $x$. This implies that the coefficients of this limiting system will be discontinuous at $x=0$ in the generic situation and the well-posedness of the corresponding martingale problem becomes a challenge. Note that the limit being well defined is an assumption in this theorem. For a list of special cases when well-posedness holds, we refer the reader to \cite{KK04} (see also \cite{BP87}).

In this paper,  we answer the question what happens when the Cesaro limits of $\sigma$ are zero while the Cesaro limits of $b$ are equal. Namely, we study the following special case of (\ref{eq:X_eqfast})-(\ref{eq:Y_eqfast}).\re
\begin{align}
\label{eq:X_eqfastagain} d\tilde{X}^{\e}(t)&=\e^{-1}\varphi(\tilde{X}^{\e}(t),\tilde{Y}^{\e}(t))dW(t),&\tilde{X}^{\e}(0)=x_0/\e,\\
\label{eq:Y_eqfastagain} d\tilde{Y}^{\varepsilon}(t)&=[b_1(\tilde{Y}^{\varepsilon}(t))+b_2(\tilde{X}^{\e}(t),\tilde{Y}^{\varepsilon}(t))]dt+\sigma(\tilde{X}^{\e}(t),\tilde{Y}^{\varepsilon}(t))dW(t),&\tilde{Y}^{\e}(0)=y_0,
\end{align}
where $b_2$ and $\sigma$ satisfy a certain decay condition in $x$ which guarantee that their Ces\'aro limits for large $|x|$ vanish so their contribution to the dynamics for large values of $|\tilde{X}^{\e}|$ is negligible. In this setting, the result of \cite{KK04} implies that $Y^{\e}(t)$ converges uniformly in probability to the solution $y(t)$ of the ordinary differential equation (ODE)
\[
\dot{y}(t)=b_1(y(t)),\qquad y(0)=y_0.
\]
In fact, it can be shown that this convergence holds in a stronger sense (see Lemma \ref{lem:Y_close_to_y}). Consequently, $Y^{\e}(t)$ can be viewed as a perturbation of $y(t)$ and it is of interest to describe the behavior of their deviation. This system can model, for example, the effect of spatial inhomogenity localized around a one co-dimensional surface (e.g. a crack in a solid) where the homogeneous structures away from the interface can possibly be different on different sides of this interface. This is the situation, for example, in groundwater hydrology when one studies diffusion in composite porous materials, see e.g. \cite{LQFG2000} and the references therein. In our scaling, the dynamics transversal to the interface is much faster than the one parallel to it.

In \cite{PGYS15}, the authors considered the simpler case when $\tilde{X}^{\e}$ is a Brownian motion independent of the noise driving the slow motion. Namely, the stochastic differential equation (SDE)
\[
d\tilde{Y}^{\e}(t)=[b_1(Y^{\e}(t))+b_2(\e^{-1}W_1(t),Y^{\e})]dt+\sigma(\e^{-1}W_1(t),Y^{\e}(t))dW_2(t),\qquad Y^{\e}(0)=y_0,
\]
where $W_1$ and $W_2$ are independent standard Brownian motions, was considered and it was shown that the first order correction to the deterministic limit is of order $\sqrt{\e}$. Moreover, a functional central limit theorem was proved for the normalized difference $\e^{-1/2}(Y^{\e}(t)-y(t))$ with the limit identified as a process that is Gaussian when conditioned on $W_1$. This case was significantly simpler than (\ref{eq:X_eqfastagain})-(\ref{eq:Y_eqfastagain}) as the difficulty arising from the discontinuity at $x=0$ does not appear and there is no feedback from the $Y^{\e}$ process to $X^{\e}$. This allowed for a proof using only elementary stochastic analysis. In the limit, the perturbation acts only when $W_1(t)=0$ (i.e., on a Cantor set of times) and the equation for the limiting solution can be expressed in terms of the local time of $W_1$ at zero.

In the current paper, we prove a similar result for the general system (\ref{eq:X_eqfastagain})-(\ref{eq:Y_eqfastagain}) using the  martingale method. We show that if $\sigma\neq 0$, the pair $(\e\tilde{X}^{\e}(t),\e^{-1/2}(\tilde{Y}^{\e}(t)-y(t)))$ converges to a non-trivial process with singular diffusion on the interface which we describe. On the other hand if $\sigma\equiv 0$, then the pair $(\e\tilde{X}^\e(t), \e^{-1}(\tilde{Y}^\e(t) - y(t)))$ converges to a nontrivial process with singular drift on the interface. Moreover, we show that in the absence of the unperturbed system ($b_1\equiv 0$), the interesting limit behavior happens on longer time scales (of order $\e^{-2}$). The long-time limit has both singular drift and singular diffusion in the sense that there is a proper diffusion process on the interface, with coefficients that are averages with respect to a sigma-finite invariant measure of the fast motion, that is sampled at the random times when the limit of the fast motion is in the origin. Once again, these times form a Cantor set and the limiting slow motion is constant for almost all times.

 It is worth pointing out that in some sense, the original result of \cite{KK04} and the results in this paper parallel the first few steps of the program contained in (\cite{FW12}, Chapter 7) when the fast process is positively recurrent (See Section 7). 

The plan of this paper is as follows. In Section \ref{sec:assump_main}, we specify the conditions on the coefficients in  (\ref{eq:X_eqfastagain})-(\ref{eq:Y_eqfastagain}), describe the limit processes, and state our convergence results. In Section \ref{sec:aux_lemma}, we derive the main technical tools needed for our analysis. In Section \ref{sec:Tightness}, we show that the families of processes under consideration are tight while in Section \ref{sec:martingale-conv}, we prove that all subsequential limits converge to the solution of the associated martingale problem. In Section \ref{sec:uniqmarting}, we prove that these martingale problems are well-posed and characterize their unique solutions as the weak solutions of certain SDE-s. In Section \ref{rem:further}, we mention some open problems.

\subsubsection*{Acknowledgements}

{\small
The authors are grateful to D.\ Dolgopyat for introducing them to the problem and to L. Koralov and D. Dolgopyat for their helpful suggestions during invaluable discussions and for reading the manuscript. We also thank P.E. Jabin for a discussion on Section \ref{rem:further}. While working on the paper, Z.\ Pajor-Gyulai was partially supported by the NSF grants number 1309084 and DMS1101635. M. Salins was partially supported by the NSF grant number 1407615.
}

%%%%%%%%%%%%%%%%%%%%%%%%%%%%%%%%%%%%%%%%%%%%%%%%%%%%%%%%
%%%%%%%%%%%%%%%%%%%%%%% SECTION 2 %%%%%%%%%%%%%%%%%%%%%%%%%%%%
%%%%%%%%%%%%%%%%%%%%%%%%%%%%%%%%%%%%%%%%%%%%%%%%%%%%%%%%%

\section{Assumptions and main results}\label{sec:assump_main}
For convenience, we will make the change of coordinates $(X^{\e}(t),Y^{\e}(t))=(\e\tilde{X}^{\e}(t),\tilde{Y}^{\e}(t))$ and consider the system
\begin{align}
\label{eq:X_eq} dX^{\e}(t)&=\varphi(\e^{-1}X^{\e}(t),Y^{\e}(t))dW(t),& X^{\e}(0)=x_0,\\
\label{eq:Y_eq} dY^{\varepsilon}(t)&=[b_1(Y^{\varepsilon}(t))+b_2(\e^{-1}X^{\e}(t),Y^{\varepsilon}(t))]dt+\sigma(\e^{-1}X^{\e}(t),Y^{\varepsilon}(t))dW(t),& Y^{\e}(0)=y_0,
\end{align}
where $W(t)$ is a $k$-dimensional Wiener process, $\varphi(x,y)$ is a $k$-dimensional vector field, $b_1(y)$ and $b_2(x,y)$ are $d$-dimensional vector fields, and $\sigma(x,y)$ is a $d\times k$-matrix. In this way, $X^\e(t)$ is a one-dimensional process and $Y^\e(t)$ is a $d$-dimensional process.

We make the following assumptions:
\begin{enumerate}
\item $\varphi$, $b_2$, and $\sigma$ are Lipschitz continuous and $b_1(y)$ is a twice continuously differentiable $d$-dimensional vector field with bounded derivatives.
\item Decay conditions hold in $x$:
\begin{equation}\label{eq:decay_cond}
\hat{b}(x):=\sup_{y\in\mathbb{R^d}}|b_2(x,y)|_{\reals^d}\in L^1(\mathbb{R}),\qquad \hat{\sigma}^2(x)=\sup_{y\in\mathbb{R^d}}\Tr\sigma\sigma^T(x,y)=\sup_{y\in\mathbb{R^d}}\sum_{i,j=1}^d\sigma_{ij}^2(x,y) \in L^1(\mathbb{R}).
\end{equation}
\item $X^{\e}$ is non-degenerate. More precisely, if $|\varphi(x,y)|^2=\sum_{i=1}^{k}\varphi_i^2(x,y)$, we require that there are numbers $0<c_1\leq c_2<\infty$ such that
\[
c_1<|\varphi(x,y)|^2<c_2\qquad\forall (x,y)\in\mathbb{R}^{1+d}
\]
\item $1/|\varphi(x,y)|^2$ has Cesaro-limits in $x$, i.e. the limits
\[
\lim_{x\to\pm\infty}\frac{1}{x}\int_0^x\frac{1}{|\varphi(u,y)|^2}du
\]
exist. Define
\begin{equation}\label{eq:apm}
a_{\pm}(x,y)=\chi_{\{x\geq 0\}}\left(\lim_{u\to\infty}\frac{1}{u}\int_0^u\frac{1}{|\varphi(v,y)|^2}dv\right)^{-1}+\chi_{\{x<0\}}\left(\lim_{u\to -\infty}\frac{1}{u}\int_0^u\frac{1}{|\varphi(v,y)|^2}dv\right)^{-1}.
\end{equation}
 Occasionally, we will use the notation $a_{+}(y)$ to denote the restriction of $a_{\pm}$ to $\{(x,y) \in \mathbb{R}^{d+1} : x>0\}$ (similarly we will use $a_-(y)$).
\end{enumerate}

It follows from Assumption 1 that the ODE
\[
\frac{dy}{dt}=b_1(y(t))\qquad y(0)=y_0,
\]
which serves as the unperturbed part of the slow motion, has a unique globally defined solution.

Before stating our main result, let us introduce the construction which we use to describe the limit processes. Let $H_+=\{(x,y): x>0\}$ and $H_{-}=\{(x,y): x<0\}$ be two half-spaces in $\mathbb{R}^{1+d}$ and let $\mathcal{L}_{\pm}$ be second order differential operators acting on functions over $H_{\pm}$ respectively. Denote the interface by $H_0=\{(x,y): x=0\}$. Also let $\beta_i(y), \alpha_{ij}(y) \in C(\mathbb{R}^d)$, $i,j = 1,...,d$, be such that $(\alpha_{ij}(y))$ is uniformly elliptic and such that the SDE
\[
dU(t)=\beta(U(t))dt+\sqrt{\alpha(U(t))}dW(t)
\]
has a unique weak solution. Let $p_{\pm}(y)\geq 0$ with $p_{+}(y) + p_-(y) = 1$.  Let $C_0(\mathbb{R}^{1+d})$ be the Banach space of bounded, continuous functions $f:\mathbb{R}^{1+d}\rightarrow \mathbb{R}$ decaying at infinity endowed with the supremum norm.

We construct an operator $\bar{\mathcal{L}}$ on $C_0(\mathbb{R}^{1+d})$ as follows. Let its domain $\mathcal{D}(\bar{\mathcal{L}})\subseteq C_0(\mathbb{R}^{1+d})$ consist of $f\in C_0(\mathbb{R}^{1+d})$ such that
\begin{itemize}
\item The restriction of $f$ to $H_+$, $H_-$, and along $H_0$ are smooth.
\item $\partial_{y_i}f(x,y)\in C^1(\mathbb{R}^{1+d})$ for $i=1,...,d$.
\item $\mathcal{L}_+f(0+,y)=\mathcal{L}_-f(0-,y)$, for all $y\in\mathbb{R}^d$.
\item $\partial_xf(x,y)$ has right and left limits as $x \to 0$ and the gluing-boundary condition
\begin{equation}\label{eq:glueing_cond}
p_+(y)\partial_xf(0+,y)-p_-(y)\partial_xf(0-,y)+\sum_{i=1}^d\beta_i(y)\partial_{y_i}f(0,y)+\frac{1}{2}\sum_{i,j=1}^d\alpha_{ij}(y)\partial_{y_iy_j}f(0,y)=0
\end{equation}
is satisfied for all $(x,y)\in\mathbb{R}^{1+d}$.
\end{itemize}
For any $f\in\mathcal{D}(\bar{\mathcal{L}})$, let $\bar{\mathcal{L}}f(x,y)=\mathcal{L}_{\pm}f(x,y)$ for $x\in H_{\pm}$ and define $\bar{\mathcal{L}}f(0,y)$ by continuity.

This is the adaptation of a special case of the well-known Wentzell type boundary conditions (see \cite{W59a}, \cite{W59b} or the survey \cite{U67}). In general, it is hard to prove that the closure of such an operator is the generator of a continuous strong Markov process $Z(t)=(X(t),Y(t))$. Nevertheless, we will show that when $p_+(y)\equiv p_-(y)\equiv 1/2$ and $\mathcal{L}_{\pm}$ has a special form, $\bar{\mathcal{L}}$ does indeed generate a Markov semigroup on $C_0(\mathbb{R}^{1+d})$. In fact, in the cases of interest we construct the corresponding Markov process probabilistically. However, we believe that the closure of such an operator $\bar{\mathcal{L}}$ is the generator of a Markov semigroup in greater generality.

Intuitively, the process $Z(t)$ can be described as follows. Inside the half-spaces $H_+$ and $H_-$, it coincides with the diffusion generated by $\mathcal{L}_{\pm}$ up to the point when it hits $H_0$. On $H_0$, $Z(t)$ follows a diffusion process for an infinitesimal amount of time, after which it leaves $H_0$ in the direction of $H_{\pm}$ with probability $p_{\pm}$. The set of times at which $Z(t)\in H_0$ has Lebesgue measure zero and consequently the displacement due to the effect of the boundary is singular in time (the support of the increments is a Cantor set in $\mathbb{R}$). To make this intuition rigorous and obtain a characterization of $Z(t)$, let us recall the definition of the local time of a semimartingale (see Theorem 7.1 in \cite{KS}). 
\begin{Definition}\label{def_local_time}
The symmetric local time of a semimartingale $S$ with quadratic variation $\langle S\rangle$ is the unique nonnegative random field
\[
L=\{L^S(t,x); (t,x)\in[0,\infty)\times \mathbb{R}\}
\]
such that the following hold:
\begin{enumerate}
\item The mapping $(t,x)\to L^S(t,x)$ is measurable and $L^S(t,x)$ is adapted.
\item For each $x\in\mathbb{R}$, the mapping $t\to L^S(t,x)$ is non-decreasing and constant on each open interval where $S(t)\neq x$.
\item For every Borel measurable $f:\mathbb{R}\to [0,\infty)$, we have
\[
\int_0^tf(S_s)d\langle S\rangle_s=\int\limits_{\mathbb{R}}f(x)L^S(t,x)dx\qquad a.s.
\]
\item $L^S(t,x)$ is a.s. jointly continuous in $t$ and $x$ for $x\neq 0$, while the one sided limits exist at $x=0$ and
\begin{equation}\label{eq:sym_loc_time}
L^S(t,0)=\frac{1}{2}(L^S(t,0+)+L^S(t,0-))
\end{equation}
\end{enumerate}
\end{Definition}

Note that instead of (\ref{eq:sym_loc_time}), we could have choosen an arbitrary convex combination. Other common choices include $L^S(t,0)=L^S(t,0+)$ (right local time) or $L^S(t,0)=L^S(t,0-)$ (left local time). For our current purposes, however, the symmetric local time satisfying (\ref{eq:sym_loc_time}) is the most suitable as it makes our formulas more transparent.

Assume that $Z(t)$ is the unique solution of the martingale problem given $\bar{\mathcal{L}}$ for appropriate $\beta$, $\alpha$ and $p_{\pm}$. It is natural to conjecture (see also e.g. \cite{HM11},\cite{L}) that it is also the unique weak solution of an SDE involving singular terms on the interface $x=0$. Namely, let
\begin{equation} \label{eq:mathcal-L-pm}
\mathcal{L}_{\pm}=\frac{1}{2}A(x,y)\partial_{xx}+B(x,y)\partial_y
\end{equation}
be the diffusion operator on $H_-\cup H_+$,\re where $\partial_y$ is the $d$-dimensional gradient with respect to the $y$ variable and $A(x,y)$, $B(x,y)$ are matrix-valued functions that are uniformly elliptic and sufficiently regular on both halfspaces $H_i$. Assume $A(x,y)$ has a jump discontinuity at $x=0$ ($B(x,y)$ may or may not be continuous at $x=0$). Then $Z(t)$ is a weak-solution to the following SDE
\begin{align}
\label{eq:SDE_X}dZ_1(t)=\sqrt{A(Z_1(t),Z_2(t))}dW(t)+[p_+(Z_2(t))-p_-(Z_2(t))]dL^{Z_1}(t),\qquad Z_1(0)=x_0,\\
\label{eq:SDE_Y}dZ_2(t)=B(Z_1(t),Z_2(t))dt+\beta(Z_2(t)) dL^{Z_1}(t)+\sqrt{\alpha(Z_2(t))}dV^{Z_1}(t),\qquad Z_2(0)=y_0,
\end{align}
where $\sqrt{.}$ is the matrix square root, $L^{Z_1}(t)=L^{Z_1}(t,0)$ is the local time of $Z_1$ at $0$, $W$ is a one-dimensional Brownian motion. $V^{Z_1}(t)$ is a continuous martingale whose quadratic variation is $\langle V^{Z_1}_i, V^{Z_1}_j\rangle_t = \delta_{ij} L^{Z_1}(t)$  and the cross-variations are $\langle V_i^{Z_1},W\rangle_t\equiv 0$ for all $i=1,...,d$. It is natural to define  the weak solution to such an SDE as a quadruple $(Z_1,Z_2,W,V^{Z_1})$ such that $W$ is a one-dimensional Brownian motion and $V^{Z_1}$ is a continuous martingale satisfying the above variation relations and that $(Z_1,Z_2)$  solves (\ref{eq:SDE_X})-(\ref{eq:SDE_Y}) in the integral sense with $(W,V^{Z_1})$ taken as the driving noise. \color{black} In this way, $Z_1$ diffuses like a skew-Brownian motion  (see \cite{L})  modulated by $Z_2(t)$, and $Z_2(t) - \int_0^t B(Z_1(s),Z_2(s))ds$ moves only when $Z_1(t)=0$, because $L^{Z_1}(t)$ (and therefore $V^{Z_1}(t)$) is constant when $Z_1\not=0$.

Because of the discontinuous coefficients and the coupled nature of the above equations, the well-posedness of such an SDE is in general unclear. We will show, just as one would expect based on the classical situation, that the weak solutions of (\ref{eq:SDE_X})-(\ref{eq:SDE_Y}) are in one to one correspondence with the solutions of the martingale problem associated to the operator $\bar{\mathcal{L}}$ as discussed above under fairly general circumstances. However, we do not need this general result for the proofs of our main results and therefore defer the proof to the appendix.

We are ready to state our main result on the asymptotic behavior of the process ($X^{\e},Y^{\e}$). As we mentioned before, $Y^{\e}(t)$ converges to $y(t)$. Our first main result concerns the rate of this convergence. Let $C([0,\infty),\mathbb{R}^k)$ be the space of all continuous $\mathbb{R}^k$-valued functions on the halfline endowed with the metric
\[
\rho(f,g)=\sum_{n=1}^{\infty}\frac{1}{2^n}\frac{\sup_{t\in[0,n]}|f(t)-g(t)|}{1+\sup_{t\in[0,n]}|f(t)-g(t)|}
\]
It is well known that weak convergence of probability measures over this space is equivalent to weak convergence of measures in $C([0,T],\mathbb{R}^k)$ with the supremum norm for all $T>0$. Let us introduce the normalized deviation process $\zeta^{\e}(t)=\e^{-1/2}(Y^{\e}(t)-y(t))$.

\begin{Theorem}\label{main_result2}
The triple $(X^{\e}(t),y(t),\zeta^{\e}(t))$ converges in law in $C([0,\infty),\mathbb{R}^{1+2d})$ to a Markov process which is the unique solution of the martingale problem associated to the operator $\bar{\mathcal{L}}$ described above with
\[
\mathcal{L}_{\pm}=\frac{1}{2}a_{\pm}(x,y)\frac{\partial^2}{\partial x^2}+b_1(y)\cdot\frac{\partial}{\partial y}+\left(\partial_y{b_1}(x,y)\cdot \zeta\right) \cdot \frac{\partial}{\partial{\zeta}}
\]
 where $a_{\pm}$ is as in (\ref{eq:apm})  and
\[
p_{\pm}(y)\equiv 1/2,\qquad \beta(y)=0,\qquad \alpha(y)= \begin{pmatrix} 0 & 0 \\ 0 & \int_{-\infty}^{\infty}\left(\frac{\sigma\sigma^T}{|\varphi|^2}\right)(x,y)dx \end{pmatrix}.
\]
Moreover, there is a one dimensional Brownian motion $W$ and a continuous martingale $V^{X^0}$ such that $\langle V_i^{X^0}, V_j^{X_0}\rangle_t= \delta_{ij} L^{X^0}$ and $\langle V_i,W\rangle=0$ for each $i=1,...,d$, and such that the  limit process of the pair $(X^{\e},\zeta^{\e})$ satisfies the inhomogeneous equations
\begin{align}\label{eq:the_solution}
dX^0(t)&=\sqrt{a_{\pm}(X^0(t),y(t))}dW(t),\\
  \label{eq:the_solution2} d\zeta^0(t)&=\partial_y b_1(X^0(t),y(t))\cdot \zeta^0(t)dt+\sqrt{\left(\int_{-\infty}^\infty\left(\frac{ \sigma\sigma^T}{|\varphi|^2}\right)(x,y(t))dx\right)} dV^{X^0}(t),
\end{align}
with $(X^0(0), \zeta^0(0))=(x_0,0)$ where $\sqrt{\cdot}$ in (\ref{eq:the_solution2}) denotes the matrix square root,
 and $\partial_y b_1(y)$ is the derivative tensor of the vector field $b_1$ at $y\in\mathbb{R}^d$, i.e $(\partial_yb_1(y))_{ij}=\partial(b_1)^{i}/\partial y^j$.
\end{Theorem}

\begin{Remark}\label{rem:varpar1}
Note that once the processes $X^0$ and $y$ are constructed, the solution of (\ref{eq:the_solution2}) is explicitly given by the variation of parameters formula
\[
\zeta^0(t)=\int_0^te^{\int_s^t (\partial_yb_1(X^0(r),y(r))dr}\sqrt{\left( \int_{-\infty}^\infty \left(\frac{\sigma\sigma^T}{|\varphi|^2}\right)(x,y(s)) dx \right)} dV^{X^0}(s).
\]
\end{Remark}

As is apparent from this theorem, the drift part of the perturbation ($b_2$) does not contribute to the deviations of order $\sqrt{\varepsilon}$. On the other hand, the following result suggests that $b_2$ does play a role in fluctuations of order $\varepsilon$. 

\begin{Theorem}\label{main_result1}
Let $\sigma\equiv 0$ and $\bar{\zeta}^{\e}(t)=\e^{-1}(Y^{\e}(t)-y(t))$. The triple $(X^{\e}(t),y(t),\bar{\zeta}^{\e}(t))$ converges in law in $C([0,\infty),\mathbb{R}^{1+{2}d})$ to the unique solution of the martingale problem associated to the operator $\bar{\mathcal{L}}$ described above with
\[
\mathcal{L}_{\pm}=\frac{1}{2}a_{\pm}(x,y)\frac{\partial^2}{\partial x^2}+b_1(y)\cdot\frac{\partial}{\partial y}+\left(\partial_y b_1(x,y)\cdot \zeta\right) \cdot\frac{\partial}{\partial{\zeta}}
\]
and
\[
p_{\pm}(y)\equiv 1/2,\qquad \beta(y)= \left(0,\int_{-\infty}^{\infty}\left(\frac{b_2}{|\varphi|^2}\right)(x,y)dx\right),\qquad \alpha(y)\equiv 0.
\]
 Moreover, there is a Brownian motion $W$ such that the limit process of the pair $(X^{\e},\zeta^{\e})$ solves the inhomogeneous equation
\begin{align*}
dX^0(t)&=\sqrt{a_{\pm}(X^0(s),y(s))}dW(t)& X^0(0)=x_0,\\
d{\zeta}^0(t)&=\partial_y{b_1}(X^0(t),y(t))\cdot{\zeta}^0(t)dt+\left(\int_{-\infty}^\infty\left(\frac{b_2}{|\varphi|^2}\right)(x,y(t))dx\right) dL^{X^0}(t) &
  {\zeta}^0(0)=0,
\end{align*}
where $\sqrt{\cdot}$ in (\ref{eq:the_solution2}) denotes the matrix square root,
 and $\partial_y b_1(y)$ is the derivative tensor of the vector field $b_1$ at $y\in\mathbb{R}^d$, i.e $(\partial_yb_1(y))_{ij}=\partial(b_1)^{i}/\partial y^j$.
\end{Theorem}
\begin{Remark}
In this case, the variation of parameters formula implies
\[
\zeta^0(t)=\int_0^te^{\int_s^t (\partial_yb_1(X^0(r),y(r))dr}\left( \int_{-\infty}^\infty \left(\frac{b_2}{|\varphi|^2}\right)(x,y(s)) dx \right) dL^{X^0}(s).
\]
\end{Remark}

When $b_1(y)\equiv 0$, we have $y(t)\equiv 0$ and it follows from Theorem \ref{main_result2} that $\sup_{t\in[0,T]}|Y^{\e}(t)|$ converges to zero for every fixed $T$ in probability. On long time scales (of order $\e^{-2}$), however, the process converges to a limit with both singular drift and singular diffusion on the interface. Let $\bar{X}^{\e}(t) = \e X^\e(\e^{-2}t)$ and $\bar{Y}^\e(t) = Y^\e(\e^{-2}t)$. This pair satisfies the SDE
\begin{align}
  &d\bar{X}^\e(t) = \varphi(\e^{-2} \bar{X}^\e(t), \bar{Y}^\e(t)) dW(t),& \bar{X}^{\e}(0)=\e x_0,\\
  &d\bar{Y}^\e(t) = \e^{-2} b_2(\e^{-2} \bar{X}^\e(t), \bar{Y}^\e(t)) dt + \e^{-1} \sigma( \e^{-2} \bar{X}^\e(t), \bar{Y}^\e(t)) dW(t),&\bar{Y}^{\e}(0)=y_0,
\end{align}
with a Brownian motion $W$ different from the one in (\ref{eq:X_eq})-(\ref{eq:Y_eq}). 

\begin{Theorem}\label{main_result3}
 The pair $(\bar{X}^{\e},\bar{Y}^{\e})$ converges in law in $\C([0,\infty); \mathbb{R}^{1+d})$ to the unique solution of the martingale problem associated to the operator $\bar{\mathcal{L}}$ with
\[
\mathcal{L}_{\pm}=\frac{1}{2}a_{\pm}(x,y)\frac{\partial^2}{\partial x^2}
\]
and
\[
p_{\pm}(y)\equiv 1/2,\qquad \beta(y)=\int_{-\infty}^{\infty}\left(\frac{b_2}{|\varphi|^2}\right)(x,y)dx,\qquad \alpha(y)=\int_{-\infty}^{\infty}\left(\frac{\sigma\sigma^T}{|\varphi|^2}\right)(x,y)dx.
\]
Moreover, the limit process is the unique weak solution of the equation
\begin{align}
\label{eq:longlimX}d\bar{X}^0(t)&=\sqrt{a_{\pm}(\bar{X}^0(t),\bar{Y}^0(t))}dW(t),\\
\label{eq:longlimY}d\bar{Y}^0(t)&=\left(\int_{-\infty}^{\infty}\left(\frac{b_2}{|\varphi|^2}\right)(x,\bar{Y}^0(t))dx\right)dL^{\bar{X}^0}(t)
+\left(\sqrt{\int_{-\infty}^{\infty}\left(\frac{\sigma\sigma^T}{|\varphi|^2}\right)(x,\bar{Y}^0(t))dx}\right)dV^{\bar{X}^0}(t),
\end{align}
with $(\bar{X}^0(0), \bar{Y}(0))=(0,y_0)$ where $L^{\bar{X}^0}$, $W$, and $V^{\bar{X}^0}$ are as in \eqref{eq:SDE_X} and \eqref{eq:SDE_Y}.
\end{Theorem}

As one can see from (\ref{eq:longlimY}), the singular drift and diffusion on the interface are both significant on these longer time scales. Also, a careful look at the explicit formula for $\beta(y)$ and $\alpha(y)$ reveals that they are essentially averages with respect to a $\sigma$-finite invariant measure of the fixed $y$ process $dX^y(t)=\varphi(X^y(t),y)dW(t)$ (with density $c/|\varphi(\cdot,y)|^2$ for any $c>0$). In this sense, Theorem \ref{main_result3} is an averaging result where the singular term on the interface follows the averaged dynamics of the slow component with respect to the invariant measure of the fast motion. Due to the null recurrence of the fast dynamics, however, the natural timescale for this averaged process is not proportional to $t$ (as it is due to the ergodic behavior of fast-slow systems in the positive recurrent case). Instead, it is proportional to the amount of time that the fast component asymptotically spends in the essential support of $b$ and $\sigma$, or in other words to the local time at $x=0$. Since our fast motion is roughly Brownian, this time scale is of order $\sqrt{t}$.

Note that even though the theorems above are stated for a flat interface,  they are applicable to situations when the interface is a more general hypersurface. We also remark that we chose the initial conditions $x_0$, $y_0$ somewhat arbitrarily. The modifications one needs to make to include the case when $x_0,y_0$ can depend on $\e$ differently are straightforward.

%%%%%%%%%%%%%%%%%%%%%%%%%%%%%%%%%%%%%%%%%%%%%%%%%%%%%%%%
%%%%%%%%%%%%%%%%%%%%%%%%% SECTION 3 %%%%%%%%%%%%%%%%%%%%%%%%%%
%%%%%%%%%%%%%%%%%%%%%%%%%%%%%%%%%%%%%%%%%%%%%%%%%%%%%%%%

\section{Auxiliary lemmas}\label{sec:aux_lemma}
First we derive an estimate on certain integral functionals of the fast process  (\ref{eq:X_eq})

%% Lemma 3.1 %%
\begin{Lemma}\label{lem:mylemma}
 For any $p \geq 1$, there exists $C_p>0$ such that for  all $\psi \in L^1(\reals)$, all stopping times $0<\tau_1<{\tau_2}$ with $\E\tau_2^{p/2}<\infty$, and all $\e>0$,
  \begin{equation} \label{continuity-for-integral-eq}
     \E \left| \frac{1}{\e} \int_{{\tau_1}}^{{\tau_2}} \psi(\e^{-1}X^\e(r)) dr \right|^p \leq  C_p |\psi|_{L^1(\reals)}^p\E|{\tau_2}-{\tau_1}|^{\frac{p}{2}}\color{black}.
  \end{equation}
\end{Lemma}

\begin{proof}
If $f \in C^2(\reals)$, then by Ito formula,
  \begin{align} \label{Ito-eq}
    f(&\e^{-1}X^\e({\tau_2})) - f(\e^{-1}X^\e({\tau_1})) = \\
\nonumber &= \frac{1}{\e} \int_{\tau_1}^{\tau_2} f'(\e^{-1}X^\e(r)) \varphi (\e^{-1}X^{\e}(r),Y^{\e}(r))\color{black} dW(r) + \frac{1}{2\e^2} \int_{\tau_1}^{\tau_2} f''(\e^{-1}X^\e(r)) |\varphi(\e^{-1}X^{\e}(r),Y^{\e}(r))|^2 \color{black} dr
  \end{align}
 almost surely. In fact, by the Meyer-Tanaka formula, the above expression is valid even if $f''$ is not continuous, but is only $L^1$.
  We want to construct $f$ such that $f''(x) = |\psi(x)|$. Since $\psi \in L^1(\mathbb{R})$, the function
  \[\Psi(x) = \int_{-\infty}^x |\psi(y)| dy\]
  is well defined, continuous and bounded ($0 \leq \Psi(x) \leq |\psi|_{L^1(\mathbb{R})})$. Then we can define
  \[f(x) = \int_0^x \Psi(y) dy.\]
  This function is Lipschitz continuous because
  \begin{equation} \label{f-Lipschitz-eq}
   |f(x_2)-f(x_1)|\leq \left|\int_{x_1}^{x_2} \Psi(y) dy \right| \leq |\psi|_{L^1(\mathbb{R})} |x_2-x_1|.
  \end{equation}

  Now, because $c_1 \leq |\varphi(x,y)|^2$,
  \[ \frac{1}{\e} \int_{\tau_1}^{\tau_2} |\psi(\e^{-1}X^\e(r))| dr \leq c {\e}\color{black}\frac{1}{2\e^2} \int_{\tau_1}^{\tau_2} |\psi(\e^{-1}X^\e(r))| |\varphi(\e^{-1}X^\e(r),Y^\e(r))|^2 dr,\]
where here and in what follows, $c$ is some (possibly different  from line to line) positive constant independent of $\e$.  Plugging this into \eqref{Ito-eq},
  \begin{align*} \frac{1}{\e} \int_{\tau_1}^{\tau_2}& |\psi(\e^{-1}X^\e(r))| dr\leq \\
&\leq c \left(\e(f(\e^{-1}X^\e({\tau_2})) - f(\e^{-1}X^\e({\tau_1})))  -\color{black} \int_{\tau_1}^{\tau_2} \Psi(\e^{-1}X^\e(r)) \varphi(\e^{-1}X^\e(r),Y^\e(r)) \color{black}dW(r)  \right)
\end{align*}
  By the elementary inequality $|a+b|^p \leq 2^{p-1}(|a|^p + |b|^p)$,
  \begin{align*}
\E &\left|\frac{1}{\e} \int_{\tau_1}^{\tau_2} \psi(\e^{-1}X^\e(r)) dr \right|^p \leq c\Bigg( \E \e^p \left|f(\e^{-1}X^\e({{\tau_2}}))-f(\e^{-1}X^\e({\tau_1}))\right|^p \\
&+ \E \left|\int_{\tau_1}^{{\tau_2}} \Psi(\e^{-1}X^\e(r)) \varphi(\e^{-1}X^\e(r),Y^\e(r)) dW(r) \right|^p \Bigg).
\end{align*}
  By \eqref{f-Lipschitz-eq},
  \begin{align*}\e | f(\e^{-1}X^\e({{\tau_2}}))& -f(\e^{-1}X^\e({\tau_1}))| \leq\\
&\leq \e |\psi|_{L^1(\reals)} |\e^{-1}X^\e({{\tau_2}}) - \e^{-1}X^\e({\tau_1})| = |\psi|_{L^1(\reals)} \left|\int_{\tau_1}^{{\tau_2}} \varphi(\e^{-1}X^\e(r),Y^\e(r)) dW(r)\right|, \end{align*}
  and therefore,
  \begin{align*}\E \left|\frac{1}{\e} \int_{\tau_1}^{{\tau_2}} \psi(\e^{-1}X^\e(r)) dr \right|^p \leq c &|\psi|_{L^1(\reals)}^p \E \left| \int_{\tau_1}^{{\tau_2}} \varphi(\e^{-1}X^\e(r), Y^\e(r)) dW(r)  \right|^p +\\
&+ c\E\left| \int_{\tau_1}^{{\tau_2}} \Psi(\e^{-1}X^\e(r)) \varphi(\e^{-1}X^\e(r),Y^\e(r)) dW(r) \right|^p \end{align*}
  Each of these expectations can be bounded by the BDG inequality, and we see that
  \begin{align*}\E \left| \frac{1}{\e} \int_{\tau_1}^{{\tau_2}} \psi(\e^{-1}X^\e(r)) dr \right|^p \leq c&\color{black}|\psi|_{L^1(\reals)}^p \left| \E \int_{\tau_1}^{{\tau_2}} |\varphi(\e^{-1}X^\e(r),Y^\e(r))|^2 \color{black} dr \right|^{p/2} + \\
& + c\left| \E \int_{\tau_1}^{{\tau_2}} \Psi^2(\e^{-1}X^\e(r)) |\varphi(\e^{-1}X^\e(r),Y^\e(r))|^2 \color{black} dr \right|^{p/2} . \end{align*}
  All of these integrands are bounded because $|\varphi|\leq c_2$ and $\Psi(x) \leq |\psi|_{L^1(\reals)}$. Therefore, we can conclude that
  \[\E \left| \frac{1}{\e} \int_{\tau_1}^{{\tau_2}} \psi(\e^{-1}X^\e(r)) dr \right|^p \leq C_p|\psi|_{L^1(\reals)}^p\E|{{\tau_2}}-{\tau_1}|^{\frac{p}{2}}.\]
as required.
\end{proof}

Next we show that the slow motion \eqref{eq:Y_eq} converges to the unperturbed system in a strong sense.

%% Lemma 3.2 %%%

\begin{Lemma}\label{lem:Y_close_to_y}
For every $p\geq 1$, there exists a constant $C_p$ such that
\[
\sup_{t\in [0,T]}\mathrm{E}|Y^{\varepsilon}(t)-y(t)|^p<C_{p}\max\{\e^pT^{p/4},\varepsilon^{p/2}T^{p/2}\}e^{C_p T^{p-1}}.
\]
\end{Lemma}

\begin{proof}
Note that
\begin{align*}
Y^{\varepsilon}(t)-y(t)&=\int_0^t(b_1(Y^{\varepsilon}(s))-b_1(y(s))ds+\int_0^tb_2\left(\e^{-1}X^{\e}(s),Y^{\varepsilon}(s)\right)ds+\\
&\int_0^t\sigma\left(\e^{-1}X^{\e}(s),Y^{\varepsilon}(s)\right)dW(s)=I_1^{\varepsilon}(t)+I_2^{\varepsilon}(t)+I_3^{\varepsilon}(t).
\end{align*}
By the Lipschitz continuity of $b_1$ and Jensen's inequality,
\[
\mathrm{E}|I_1^{\varepsilon}(t)|^p\leq T^{p-1}\textnormal{Lip}(b_1)^p\int_0^t\mathrm{E}|Y^{\varepsilon}(s)-y(s)|^p.
\]
On the other hand, recall that $\hat{b}(x)=\sup_{y\in\mathbb{R}^d}|b_2(x,y)|$ and note \color{black}
\[
\sup_{t\in[0,T]}\mathrm{E}|I^{\varepsilon}_2(t)|^p\leq\mathrm{E}\left(\int_0^T\hat{b}(\e^{-1}X^{\e}(s))ds\right)^p\leq C_p \varepsilon^p T^{p/2},
\]
where in the last inequality we used Lemma \ref{lem:mylemma} with $\tau_1=0$ and $\tau_2=T$. Finally, it is easy to see that the scalar quadratic variation of $I_3$ is
\[
\langle I_3\rangle_t=\int_0^t\Tr\sigma\sigma^T(\e^{-1}X^{\e}(s),Y^{\e}(s))ds\leq\int_0^t\hat{\sigma}^2(\e^{-1}X^{\e}(s))ds,
\]
(where we recall $\hat{\sigma}^2(x)=\sup_{y\in\mathbb{R}^d}\Tr\sigma\sigma^T(x,y)$) and therefore by the Burkholder-Davis-Gundy inequality and Lemma \ref{lem:mylemma} we have
\begin{align}\label{eq:I3}
\sup_{t\in[0,T]}\mathrm{E}|I_3^{\varepsilon}(t)|^p&\leq  C_p\mathrm{E}{\left(\int_0^T\hat{\sigma}^2(\e^{-1}\color{black}X^{\e}(s))ds\right)}^{p/2}<C_{p}\varepsilon^{p/2}T^{p/4}.
\end{align}
The result now follows from Gronwall's lemma.
\end{proof}

Next, we recall Lemma 3.5 from \cite{KK04}, which explains why the Cesaro limits of $\frac{1}{|\varphi|^2}$ describe the limiting behavior of $X^\e$.
\begin{Lemma}\label{lem:conv_of fast_motion}
For any $f \in C(\reals_+\times \reals \times \reals^d)$ and any $T<\infty$, we have
\begin{equation}\label{eq:KKlemma}
\sup_{t\in[0,T]}\left|\int_0^{t}f(s,X^\e(s),Y^\e(s))(|\varphi|^2(\e^{-1}X^{\e}(s),Y^{\e}(s))-a_\pm( X^{\e}(s),Y^{\e}(s)))ds\right|\rightarrow 0,
\end{equation}
as $\e\downarrow 0$ in probability where $a_\pm$ is defined in \eqref{eq:apm}
Moreover, (\ref{eq:KKlemma}) holds for every family of stopping times $\{\tau^{\e}\}_{\e\leq\e_0}$ in place of $T$ such that $\{X^{\e}(\tau^{\e})\}_{\e\leq\e_0}$ is tight.
\end{Lemma}

Before proving Lemma \ref{lem:conv_of fast_motion}, we establish the following fact about Cesaro limits.
\begin{Lemma} \label{lem:Cesaro-times-f}
  Assume that $g: \reals \to \reals$ is bounded and has the Cesaro limit
  \[\bar{g}:= \lim_{x \to +\infty} \frac{1}{x} \int_0^x g(\xi)d\xi. \]
  Then for any $f \in L^1_{\textnormal{loc}}(\reals)$, and $R>0$,
  \begin{equation}
    \lim_{\e \to 0} \sup_{x\in[0,R]} \left| \bar{g} \int_0^x f(\xi) d\xi - \int_0^x {f(\xi)}{g(\e^{-1}\xi)} d\xi \right| = 0.
  \end{equation}
\end{Lemma}
\begin{proof}
  First, we consider the case where $f\equiv 1$ (see \cite{KK04}). In this case, for $x \in \reals$,
  \begin{equation} \label{eq:Cesaro-antiderivative}
    \int_0^x (\bar{g} - g(\e^{-1}\xi) )d \xi = x \left( \bar{g} - \frac{\e}{x} \int_0^{x/\e} g(\xi) d\xi \right) =: x \alpha\left( \frac{x}{\e} \right).
  \end{equation}
  Because $\bar{g}$ is the Cesaro limit of $g$, for any fixed $x$, the above expression converges to zero as $\e \to 0$.
  This convergence is uniform in $|x| \leq R$ because for $|x| \leq R$,
  \begin{equation} \label{eq:Cesaro-uniform}
    \left|x\alpha \left(\frac{x}{\e}\right)\right|= \left| \int_0^x (\bar{g} - g(\e^{-1} \xi)) d\xi \right| \leq \min \left\{ 2x |g|_{L^\infty}, R \left|\bar{g} - \frac{\e}{x} \int_0^{x/\e} g(\xi) d\xi \right| \right\}.
  \end{equation}

  Next, we consider the case where $f \in C^1$. In this case, by integrating by parts,
  \begin{equation}
    \int_0^x f(\xi)( \bar{g} - g(\e^{-1}\xi)) d\xi = f(x) x \alpha\left(\frac{x}{\e}\right) - \int_0^x f'(\xi) \xi \alpha \left( \frac{\xi}{\e}\right) d\xi.
  \end{equation}
  Therefore,
  \begin{equation}
    \sup_{0\leq x\leq R} \left| \int_0^x f(\xi) (\bar{g} - g(\e^{-1}\xi)) d\xi \right| \leq  \sup_{0\leq x \leq R} \left|x \alpha\left(\frac{x}{\e}\right) \right| \left( |f|_{C^1[0,R]}\right)(1+R)
  \end{equation}
  and it follows from \eqref{eq:Cesaro-uniform} that the convergence is uniform.

  Finally, for a general $f \in L^1_{\textnormal{loc}}(\reals)$ and $R>0$, we can approximate $f$ in $L^1([0,R])$ by a sequence $\{f_n\} \subset C^1([0,R])$. Then for $n \in \nat$,
  \[
    \sup_{0\leq x\leq R} \left| \int_0^x f(\xi)(\bar{g} - g(\e^{-1}\xi)) d\xi \right| \leq 2 |g|_{L^\infty([-R,R])} |f-f_n|_{L^1([-R,R])}
    + \sup_{0 \leq x \leq R} \left| \int_0^x f_n(\xi)(\bar{g} - g(\e^{-1}\xi)) d\xi \right|.
  \]
  By first choosing $n$ large and then $\e$ small we can make the above expression arbitrarily small.
\end{proof}
\begin{proof}[Proof of Lemma \ref{lem:conv_of fast_motion}]
 Very similar to the proof of Lemma 3.5 in \cite{KK04}, we only include a sketch of the proof. For $\e>0$, define the auxiliary function $u^\e(t,x,y)$ to solve
  \begin{equation}
  \begin{cases}
    |\varphi(\e^{-1}x,y)|^2 u^\e_{xx}(t,x,y) = f(t,x,y)( |\varphi(\e^{-1}x,y)|^2 - a_\pm(x,y)).\\
    u^\e(t,0,y)= u_x^\e(t,0,y) = 0, \ \ \lim_{\e \to 0}\sup_{|x|\leq R} |u^\e(t,x,y)| = 0.
  \end{cases}
  \end{equation}

  By dividing by $|\varphi|^2$ and integrating, we see that for $x>0$  (the situation for $x<0$ is similar),
  \[u^\e_x(t,x,y) = \int_0^x f(t,\xi,y)d\xi - a_+(y)\int_0^x \frac{f(t,\xi,y)}{|\varphi(\e^{-1}\xi,y)|^2}d\xi. \]
  By Lemma \ref{lem:Cesaro-times-f}, this converges to zero as $\e \to 0$ uniformly in $|x| \leq R$ if and only if $a_+(y)$ is the reciprocal of the Cesaro limit of $\frac{1}{|\varphi|^2}$. By integrating in $x$ once more, we show that $u^\e$ also converges to zero as $\e \to 0$ uniformly in $|x|\leq R$.
  The proof concludes by applying Ito formula to $u_\e(t,X^\e(t),Y^\e(t))$, identifying the integral in (\ref{eq:KKlemma}) as the Ito-correction term, a standard localization argument and estimating the resulting other terms. The last statement of the lemma follows similarly using that $X^{\e}(\tau^{\e})$ is in a large enough compact set for all $\e<\e_0$ with high probability.
\end{proof}

 Finally, we state the version of the Ito-Tanaka-Meyer formula that is relevant to our situation (a special case of \cite{P07}).

\begin{Lemma}\label{lem:itm}
Let $(Z_1,Z_2,W,V^{Z_1})$ be a weak solution of (\ref{eq:SDE_X})-(\ref{eq:SDE_Y}) with $p_+=p_-=1/2$ and assume that $f\in C(\mathbb{R}^{1+d})$ has continuous first and second order derivatives in $y$  and for fixed $y\in \mathbb{R}^d$, $f(.,y)\in C^1(\mathbb{R}_+\cup\mathbb{R}_-)$. Then
\begin{align*}
f(&Z_1(t),Z_2(t))=f(Z_1(0),Z_2(0))+\int_0^t\partial_x^{sym}f(Z_1(s),Z_2(s))\sqrt{A(Z_1(s),Z_2(s))}dW(s)\\
&+\sum_{i=1}^d\int_0^t\partial_{y_i}f(Z_1(s),Z_2(s))\sqrt{\alpha(Z_2(s))}dV^{Z_1}(s)
\\
&+\int_0^t \Bigg[p_+(Y(s))f_x(0+,Z_2(t))-p_-(Y(s))f_x(0-,Z_2(t))\\
 &+ \sum_{i=1}^d  \partial_{y_i}f(Z_1(s),Z_2(s))\beta_i(Z_2)
+\frac{1}{2}\sum_{i,j=1}^d\partial_{y_iy_j}f(Z_1(t),Z_2(t))\alpha_{ij}(Z_2(t)) \Bigg]dL^{Z_1}(t)\\
&+\frac{1}{2}\int_0^t\partial_{xx}f(Z_1(s),Z_2(s))A(Z_1(s),Z_2(s)) ds+\int_0^t \sum_{i=1}^d \partial_{y_i} f(Z_1(s),Z_2(s))B(Z_1(s),Z_2(s)) ds,
\end{align*}

where used the symmetrized $x$-derivative
\[
\partial_x^{sym}f(x,y)=\frac{\partial_xf(x+,y)+\partial_xf(x-,y)}{2}.
\]
\end{Lemma}
 Note that the symmetrized derivative was necessary by our choice to use the symmetric local time.
\color{black}

%%%%%%%%%%%%%%%%%%%%%%%%%%%%%%%%%%%%%%%%%%%%%%%%%%%%%%%%%
%%%%%%%%%%%%%%%%%%%%%%%%%% SECTION 4 %%%%%%%%%%%%%%%%%%%%%%%%%%
%%%%%%%%%%%%%%%%%%%%%%%%%%%%%%%%%%%%%%%%%%%%%%%%%%%%%%%%%

\section{Tightness}\label{sec:Tightness}
In this section we are going to show that the laws of the families in Theorem \ref{main_result2}, Theorem \ref{main_result1}, and Theorem \ref{main_result3} are tight in the space $C([0,\infty);\reals^{1+d})$.
To prove tightness, we will show that there exists $c>0$, $p> 0$, $q> 1$ such that for all $T>0$ and $0\leq s \leq t\leq T$, and $\e>0$,
\[\E|X^\e(t) - X^\e(s)|^p \leq c |t-s|^q, \]
\[\E |\zeta^\e(t) - \zeta^\e(s)|^p \leq c |t-s|^q\]
and in the long-time case
\[\E|\bar{X}^\e(t) - \bar{X}^\e(s)|^p \leq c |t-s|^q,\]
\[\E |\bar{Y}^\e(t) - \bar{Y}^\e(s)|^p \leq c |t-s|^q.\]
All of these results are consequences of Lemma \ref{lem:mylemma}.
\begin{Theorem} \label{thm:tightness}
  Let $(X^\e, \zeta^\e)$ satisfy the conditions of Theorem \ref{main_result2} or Theorem \ref{main_result1}. Then the family of laws of the pairs $\{(X^\e,\zeta^\e)\}_{\e>0}$ are tight. Let $(\bar{X}^\e, \bar{Y}^\e)$ satisfy the conditions of Theorem \ref{main_result3}, then the family of laws of the pairs $\{(\bar{X}^\e,\bar{Y}^\e)\}_{\e>0}$ are tight.
\end{Theorem}

\begin{proof}
  We prove tightness in the context of Theorem \ref{main_result2}, as the tightness for the other systems can be proved analogously. In this case,
  \[X^\e(t) - X^\e(s) = \int_s^t \varphi(\e^{-1}X^\e(r),Y^\e(r))dW(r).\]
  By the BDG inequality and the fact that $|\varphi|^2 < c_2$,
  \[\E\color{black}|X^\e(t) - X^\e(s)|^p \leq c_2 |t-s|^{p/2}. \]
   As for $\zeta^\e$, because $Y^\e(t) = y(t) + \sqrt{\e}\zeta^\e(t)$,
  \begin{align*}
    \zeta^\e(t) - \zeta^\e(s) &=  \int_s^t \left( \frac{b_1(y(r) + \sqrt{\e}\zeta^\e(r)) - b_1(y(r))}{\sqrt{\e}} \right) dr
    + \e^{-1/2} \int_s^t b_{2\color{black}}(\e^{-1} X^\e(r), y(r) + \sqrt{\e}\zeta^\e(r)) dr\\
    &\qquad+ \e^{-1/2} \int_s^t \sigma(\e^{-1}X^\e(r) , y(r) + \sqrt{\e}\zeta^\e(r)) dW(r)= I_1^\e(s,t) + I_2^\e(s,t) + I_3^\e(s,t).
  \end{align*}
  We estimate $I_2^\e(s,t)$ and $I_3^\e(s,t)$ by using Lemma \ref{lem:mylemma}.
  By \eqref{eq:decay_cond} and Lemma \ref{lem:mylemma}
  \begin{equation} \label{eq:tightness1}
    \E|I_2^\e(s,t)|^p \leq \E \left|\e^{-1/2} \int_0^t \hat{b}(\e^{-1}X^\e(r)) dr \right|^{p} \leq c \e^{1/2}|t-s|^{p/2},
  \end{equation}
  and similarly by the BDG inequality,
  \begin{equation} \label{eq:tightness2}
    \E|I_3^\e(s,t)|^p \leq c \E\left(\e^{-1} \int_s^t \hat{\sigma}^2(\e^{-1} X^\e(r)) dr \right)^{p/2} \leq c |t-s|^{p/4}.
  \end{equation}
Finally, if $p>1$, Lemma \ref{lem:Y_close_to_y} implies
  \begin{equation} \label{eq:tightness3}
    \E|I_1^\e(s,t)|^p \leq c_{p,T} \|b_1\|^{p}_{C^1(\reals^{d+1})} |t-s|^p.
  \end{equation}
  We choose $p > 4$, $q=p/4$ and tightness follows from \eqref{eq:tightness1}, \eqref{eq:tightness2}, and \eqref{eq:tightness3}  and the Kolmogorov criterion above.
\end{proof}

As a result of tightness, for every subsequence $\e_n \to 0$ there exists a $\e_{n_k}$ such that $(X^{\e_{n_k}},\zeta^{\e_{n_k}})$ converges in distribution. In the next sections we characterize the limit and show that it is the same process for all such subsequences. Consequently, the entire sequence converges in law to that unique limit.

%%%%%%%%%%%%%%%%%%%%%%%%%%%%%%%%%%%%%%%%%%%%%%%%%%%%%%%%%%
%%%%%%%%%%%%%%%%%%%%%%%%%%% SECTION 5 %%%%%%%%%%%%%%%%%%%%%%%%%
%%%%%%%%%%%%%%%%%%%%%%%%%%%%%%%%%%%%%%%%%%%%%%%%%%%%%%%%%%

\section{Convergence to the martingale problem} \label{sec:martingale-conv}

\subsection{Freidlin-Wentzell type result}
We are  going to show that the processes in Theorems \ref{main_result2}, \ref{main_result1}, and \ref{main_result3} each converge to a solution of a martingale problem. In all three cases, the main tool in proving convergence is a Freidlin-Wentzell type result (\cite{FW93}, see also \cite{HM11}). Assume we have a family of $(1+k)$-dimensional processes $Z^{\e}$. Depending on the situation, $Z^\e$ will be either $Z^\e=(X^{\e}, y, \zeta^{\e})$ ($k=2d$) or $Z^\e=(\bar{X}^\e, \bar{Y}^\e)$ ($k=d$). We are going to use the notation $\E_z(\cdot)$ to the expectation when $Z^{\e}$ has initial point $z=(x,y,\zeta)$ or $z=(x,y)$ respectively. Recall that $H_{\pm}=\{z: \pm x >0\}$

%% THE FREIDLIN-WENTZELL THEOREM %%%
\begin{Theorem} \label{thm:F-W-1993}
Let $\bar{\mathcal{L}_i}$ be second order differential operators on $H_i$, $i=\{+,-\}$ and let $D_i\subseteq C_b^{\infty}(H_i)$ be a set of test functions that have bounded derivatives of all order. Let $\tau^{\e}=\inf\{t\geq 0: |X^{\e}(t)|\leq l(\e)\}$ for some $l(\e)\downarrow 0$ as $\e\downarrow 0$. Assume that there is a $\lambda_0\geq 0$ such that for each $\lambda>\lambda_0$, the following assumptions hold.
\begin{enumerate}
\item We have for every $f\in D_i$
\begin{equation}\label{eq:Property1}
\E_z\left[ e^{-\lambda\tau^{\e}}f(Z^{\e}(\tau^{\e}))-f(z)+\int_0^{\tau^{\e}}e^{-\lambda t}\left(\lambda f(Z^{\e}(t))-\bar{\mathcal{L}}_if(Z^{\e}(t))\right)dt\right]=\mathcal{O}(k(\e))
\end{equation}
for some $k(\e)\downarrow 0$ as $\e\downarrow 0$ uniformly for $z\in H_i$.

\item For $i\in\{+,-\}$, there is a $u_{i,\lambda}\in D_i$ such that
\[
\bar{\mathcal{L}}_iu_{i,\lambda}=\lambda u_{i,\lambda}\qquad\forall z\in H_i, |x|<1
\]
and such that $u_{\pm,\lambda}(0,y)=1$ and $\pm\partial_x u_{\pm,\lambda}(0,y)<c<0$.

\item There is a $\delta(\e)\downarrow 0$ as $\e\downarrow 0$ such that $\delta(\e)/l(\e)\to\infty$, $\delta(\e)/k(\e)\to\infty$ and
\begin{equation}\label{eq:Property3}
\sup_{z\in\mathbb{R}^{1+d}}\E_z\int_0^{\infty}e^{-\lambda t}\chi_{(-\delta(\e),\delta(\e))}(X^{\e}(t))dt\rightarrow 0
\end{equation}
as $\e\downarrow 0$.

\item Define
\begin{equation}
   \theta^\delta = \inf\{t\geq0: |X^\e(t)|\geq \delta(\e) \}.
\end{equation}

There are $p_+(y)$, $p_-(y)>0$ such that $p_++p_-=1$ and
\[
\P_z(Z^{\e}( \theta^\delta)\in H_i)\rightarrow p_i(y)
\]
as $\e\downarrow 0$ uniformly in $z\in \mathbb{R}^{1+d}$, $|x|<l(\e)$.

\item There are $\beta_i(y),\alpha_{ij}(y)$ for $i,j=1,...,d$ such that
\[
\frac{1}{\delta}\E_z(Y_i^{\e}( \theta^\delta)-y_i)\rightarrow\beta_i(y)\qquad \frac{1}{\delta}\E(Y_i^{\e}( \theta^\delta)-y_i)(Y_j^{\e}( \theta^\delta)-y_j)\rightarrow\alpha_{ij}(y),
\]
as $\e\downarrow 0$, uniformly for $|x|<l(\e)$. Moreover, we have $\sup_{|x|<l(\e)}\E|Z^{\e}( \theta^\delta)-z|^{3}=o(\delta)$.
\end{enumerate}
Then for every $f\in C_0(\mathbb{R}^{1+d})$ such that $f|_{H_i}\in D_i$ and the boundary-gluing conditions \eqref{eq:glueing_cond} hold, we have
\[
\textrm{ess sup}\left|\E_z\left[\int_{t_0}^{\infty}e^{-\lambda t}\left[\lambda f(Z^{\e}(t))-\bar{\mathcal{L}}f(Z^{\e}(t))\right]dt-e^{-\lambda t_0}f(Z^{\e}(t_0))|\mathcal{F}_{t_0}\right]\right|\rightarrow 0,
\]
as $\e\downarrow 0$ for every $t_0\geq 0$, $\lambda>\lambda_0$, where $\mathcal{F}_{t_0}=\sigma(Z^{\e}(s); s\in [0,t_0])$. In particular, for all such functions, every weak limit of $Z^{\e}$ as $\e\downarrow 0$ satisfies the martingale problem for $\bar{\mathcal{L}}$ defined by $\bar{\mathcal{L}}f(z)=\bar{\mathcal{L}}_if(z)$ for $z\in H_i$.
\end{Theorem}

\begin{proof}
By the Markov property of $(Z^{\e},\P_z)$, it is sufficient to show that
\[
\Delta(\e, z)=\E_z\int_0^{\infty}e^{-\lambda t}\left[\lambda f(Z^{\e}(t))-\bar{\mathcal{L}}f(Z^{\e}(t))\right]dt-f(z)\to 0
\]
as $\e\downarrow 0$ uniformly with respect to $z\in\mathbb{R}^{1+d}$.

Introduce the sequence of stopping times $\sigma_0=0$, $\phi_n=\inf\{t>\sigma_n: |X^{\e}(t)|=l(\e)\}$ and $\sigma_{n+1}=\inf\{t>\phi_n: |X^{\e}(t)|>\delta(\e)\}$. Using this sequence, we can decompose the above expression to terms corresponding to downcrossings (on $[\sigma_n,\phi_n]$) and upcrossings (on $[\phi_n,\sigma_{n+1}]$) respectively. More precisely,
\begin{align*}
\Delta(\e, z)&=\sum_{n=0}^{\infty}\E_z\left[e^{-\lambda\phi_n}f(Z^{\e}(\phi_n))-e^{-\lambda\sigma_n}f(Z^{\e}(\sigma_n))+\int_{\sigma_n}^{\phi_n}e^{-\lambda t}\left[\lambda f(Z^{\e}(t))-\bar{\mathcal{L}}f(Z^{\e}(t))\right]dt\right]+\\
&+\sum_{n=0}^{\infty}\E_z\left[e^{-\lambda\sigma_{n+1}}f(Z^{\e}(\sigma_{n+1}))-e^{-\lambda\phi_n}f(Z^{\e}(\phi_n))+\int_{\phi_n}^{\sigma_{n+1}}e^{-\lambda t}\left[\lambda f(Z^{\e}(t))-\bar{\mathcal{L}}f(Z^{\e}(t))\right]dt\right]
\end{align*}

Note that by the strong Markov property, we can bound the downcrossing terms by
\begin{align}
\nonumber \sum_{n=0}^{\infty}\E_ze^{-\lambda\sigma_n}\E_{z=Z^{\e}(\sigma_n)}\left[e^{-\lambda\tau^{\e}}f(Z^{\e}(\tau^{\e}))-f(z)+\int_0^{\tau^{\e}}e^{-\lambda t}\left[\lambda f(Z^{\e}(t))-\bar{\mathcal{L}}_{\pm}f(Z^{\e}(t))\right]dt\right]\leq\\
\label{eq:downcrossing_terms}\leq\mathcal{O}(k(\e))\sum_{n=0}^{\infty}\E e^{-\lambda\sigma_n}
\end{align}
where we used \eqref{eq:Property1}. To bound the sum on the right hand side, first note that
\begin{equation}\label{eq:sigma_n_sum_expect}
\sum_{n=0}^{\infty}\E_ze^{-\lambda\sigma_n}  \leq 1+\sum_{n=1}^{\infty}\E_ze^{-\lambda\sigma_{n-1}}\E_{Z^{\e}(\sigma_{n-1})}e^{-\lambda\tau^{\e}}
\end{equation}
Let $|x|=\delta(\e)$ and apply \eqref{eq:Property1} to $u_{i,\lambda}$ given by Condition $2.$ to get
\[
\E_ze^{-\lambda\tau^{\e}}u_{i,\lambda}(Z^{\e}(\tau^{\e}))-u_{i,\lambda}(z)=\mathcal{O}(k(\e)).
\]
This implies that
\begin{equation}\label{eq:taueps_expect}
\E_ze^{-\lambda\tau^{\e}}=\E_ze^{-\lambda\tau^{\e}}(1-u_{i,\lambda}(Z^{\e}(\tau^{\e})))+u_{i,\lambda}(z)+\mathcal{O}(k(\e))
\end{equation}
Since $|X^{\e}(\tau^{\e})|=l(\e)$, we have $|1-u_{i,\lambda}(Z^{\e}(\tau^{\e})|=\mathcal{O}(l(\e))$ whereas the condition on the derivative of $u_{i,\lambda}$ normal to the interface implies $u_{i,\lambda}(z)\leq 1-C\delta(\e)$ for some $C>0$ since $|z|=\delta(\e)$. Plugging these back to \eqref{eq:taueps_expect} gives
\[
\sup_{|z|=\delta(\e)}\E_ze^{-\lambda\tau^{\e}}\leq 1-C'\delta(\e)
\]
for small enough $\e$ for some $C'>0$.  This, together with \eqref{eq:sigma_n_sum_expect}, implies that $\sum_{n=0}^{\infty}\E_ze^{-\lambda\sigma_n}=\mathcal{O}(1/\delta(\e))$ which in turn implies that the bound in \eqref{eq:downcrossing_terms} converges to zero.

To bound the upcrossing terms, first note that the sum of the integrals $\int_{\phi_n}^{\sigma_{n+1}}$ can be bounded by
\[
||(\lambda -\bar{\mathcal{L}})f||_{\infty}\E_z\int_0^{\infty}e^{-\lambda t}\chi_{(-\delta(\e),\delta(\e))}(Z^{\e}(t))dt\rightarrow 0
\]
as $\e\downarrow 0$ uniformly in $z$  by \eqref{eq:Property3}, where we used that $f$ has bounded derivatives of all order. The remaining terms can be written as
\begin{equation}\label{eq:remaining_terms}
\sum_{n=0}^{\infty}\E_z(e^{-\lambda\sigma_{n+1}}-e^{\lambda\phi_n})f(Z^{\e}(\sigma_{n+1}))+\sum_{n=0}^{\infty}e^{-\lambda\phi_n}(f(Z^{\e}(\sigma_{n+1}))-f(Z^{\e}(\phi_n))).
\end{equation}
where the first term can immediately be bounded by
\[
||f||_{\infty}\sup_{z\in\mathbb{R}^{1+d}}\E_z\int_0^{\infty}\lambda e^{-\lambda t}\chi_{(-\delta(\e),\delta(\e))}(Z^{\e}(t))dt\to 0.
\]
For the second term in \eqref{eq:remaining_terms}, we can use the strong Markov property to write
\begin{equation}\label{eq:last_term_FWthm}
\sum_{n=0}^{\infty}\E_ze^{-\lambda\phi_n}\E_{z'=Z^{\e}(\phi_n)}\left(f(Z^{\e}( \theta^\delta))-f(z)\right).
\end{equation}
Using $|x'|=l(\e)$ and Taylor's formula, we get
\begin{align*}
\E_{z'}&\left(f(Z^{\e}( \theta^\delta))-f(z)\right)=E_{z'}[f(Z^{\e}( \theta^\delta))-f(0,y')]+\mathcal{O}(l(\e))=\\
&=\delta\P_z(Z^{\e}( \theta^\delta)\in H_+)\partial_xf(0+,y')-\delta\P_z(Z^{\e}( \theta^\delta)\in H_-)\partial_xf(0-,y')+\\
&+\delta\sum_{i=1}^d\partial_if(0,y')\frac{1}{\delta}\E_{z'}(Y^{\e}_i( \theta^\delta)-y'_i)+\delta\sum_{i,j=1}^d\partial_i\partial_jf(0,y')\frac{1}{\delta}E_{z'}(Y_i^{\e}( \theta^\delta)-y'_i)(Y_j^{\e}( \theta^\delta)-y'_j)+\\
&+o(\delta),
\end{align*}
where we used the last statement of condition 5. Now using condition 4., the rest of condition 5. and the gluing conditions, we get that this entire expression is $o(\delta)$ uniformly in $|x'|=l(\e)$. Noting that
\[
\sum_{n=0}^{\infty}\E_ze^{-\lambda\phi_n}\leq\sum_{n=0}^{\infty}e^{-\lambda\sigma_n}=\mathcal{O}(1/\delta),
\]
\eqref{eq:last_term_FWthm} converges to zero and the proof is finished.
That every weak limit of $Z^{\e}$ converges to the solution of the martingale problem of $\bar{\mathcal{L}}$ follows exactly as in \cite{FW93} after noticing that if the Laplace transform of a function is identically zero for sufficiently large $\lambda$ then the function itself is identically zero.
\end{proof}
%% END OF FREIDLIN WENTZELL PROOF %%

 We are going to need some estimates on the exit time $ \theta^\delta$.

\begin{Lemma}\label{lem:taudeltaestim}

\[
\E_z \theta^\delta\leq C\delta^2,\qquad\E_z( \theta^\delta)^2\leq C\delta^4.
\]
Moreover, there is a $\delta_0>0$ such that $\delta\in (0,\delta_0]$ and $|x|<\delta$ implies
\[
\E_z e^{u \theta^\delta}\leq \frac{1}{\cos(Cu\delta)}.
\]
\end{Lemma}

\begin{proof}
Let  $W$ be a Brownian motion such that $X^{\e}(\cdot)$ and $x+W(\int_0^\cdot |\varphi|^2(s)ds)$ have the same distribution and therefore so do $ \theta^\delta_W$ (the hitting time of $\delta$ by $W$) and $\int_0^{ \theta^\delta}|\varphi(s)|^2ds$. The results follow from this, $|\varphi|^2\in[c_1,c_2]$ and well-known Brownian formulas.
\end{proof}

\subsection{Martingale Convergence for the long-time asymptotics}
{
Now we show that the system of equations in Theorem \ref{main_result3} satisfies the conditions of Theorem \ref{thm:F-W-1993} with $D_i=\mathcal{C}_b^{\infty}(H_i)$. As $b_1\equiv 0$, we denote $b\equiv b_2$ in this case. Let $Z^\e=(\bar{X}^\e, \bar{Y}^\e)$ solve
\begin{equation}
 \begin{cases}
  d\bar{X}^\e(t) = \varphi(\e^{-2}\bar{X}^\e(t), \bar{Y}^\e(t)) dW(t),\\
  d\bar{Y}^\e(t) = \e^{-2}b(\e^{-2}\bar{X}^\e(t), \bar{Y}^\e(t)) dt + \e^{-1}\sigma(\e^{-2}\bar{X}^\e(t),\bar{Y}^\e(t))dW(t).
 \end{cases}
\end{equation}

\begin{Lemma}\label{lem:longtimeconv5}
Let $f\in D_i$ $i=-,+$. If we choose $l(\e)$ so that $l(\e) \to 0$ and $\e^{-2}l(\e) \to +\infty$, then \eqref{eq:Property1} is satisfied for $Z^\e = (\bar{X}^\e,\bar{Y}^\e)$ and $\bar{\mathcal{L}}f(x,y) = \frac{1}{2} a_{\pm}(x,y) f_{xx}(x,y)$.
\end{Lemma}

\begin{proof}
  By Ito's formula and a standard localization argument,
  \[\begin{array}{l}
  \ds{K(\e,z)=\E_{z} \left[e^{-\lambda \tau^\e} f(Z^\e(\tau^\e)) - f(z)  - \int_0^{\tau^\e} e^{-\lambda t}\left(\lambda f(Z^\e(t)) - \frac{1}{2}a_\pm^2\color{black}(Z^\e(t)) f_{xx}(Z^\e(t)) \right)dt\right]}\\
  \vs
  \ds{= \e^{-2} \E_{z} \int_0^{\tau^\e} e^{-\lambda t} \left(  \sum_{i=1}^d b_i(\e^{-2}\bar{X}^\e(t), \bar{Y}^\e(t))  f_{y_i}(Z^\e(t))  \right) dt}\\
  \ds{+ \frac{1}{2} \e^{-2} \E_{z} \int_0^{\tau^\e} e^{-\lambda t}\left( \sum_{i,j=1}^d (\sigma\sigma^T)_{ij}(\e^{-2} \bar{X}^\e(t), \bar{Y}^\e(t)) f_{y_i y_j}(Z^\e(t))\right)dt}\\
  \ds{+\e^{-1}\E_{z} \int_0^{\tau_\e} e^{-\lambda t} \left( \sum_{i=1}^d (\sigma \varphi^T)_{i}(\e^{-2}\bar{X}^\e(t),\bar{Y}^\e(t)) f_{x y_i}(Z^\e(t)) \right) dt}\\
   \ds{+  \frac{1}{2} \E_{z} \int_0^{\tau^\e} e^{-\lambda t} \left( |\varphi|^2(\e^{-2}\bar{X}^\e(t),\bar{Y}^\e(t)) - a_\pm(Z^\e(t)) \right)f_{xx}(Z^\e(t)) dt .}\\
   \ds{:= I_1^\e + I_2^\e +I_3^\e + I_4^\e.}
  \end{array}\]

  We estimate each term separately. Note that
  \[
   \ds{I^\e_1 \leq c \e^{-2} \E_{z}\color{black}\int_0^{\tau^\e} e^{-\lambda t} \hat{b}(\e^{-2} \bar{X}^\e(t)) \|f\|_{C^1} dt,}\\
  \]
 then because $\tau^\e = \inf\{t>0: |\bar{X}^\e| \leq l(\e)\}$, we have
  \[\begin{array}{l}
    \ds{I^\e_1 \leq c\e^{-2}\|f\|_{C^1} \E_{z}\color{black}\int_0^\infty e^{-\lambda t} \hat{b}(\e^{-2}\bar{X}^\e(t)) \mathbbm{1}_{\{|\bar{X}^\e(t)|>l(\e)\}} dt}\\
    \ds{\leq c \e^{-2}\|f\|_{C^1} \sum_{k=0}^\infty e^{-\lambda k} \E_{z}\color{black}\int_k^{k+1} \hat{b}(\e^{-2} \bar{X}^\e(t))\mathbbm{1}_{\{|\bar{X}^\e(t)|>l(\e)\}} dt}.
    \end{array}\]
  By Lemma \ref{lem:mylemma}, there exists a constant $c$, independent of $k$ and $\e$ such that
  \[\e^{-2} \E_{z}\color{black}\int_k^{k+1} \hat{b}(\e^{-2}\bar{X}^\e(t))\mathbbm{1}_{\{|\bar{X}^\e(t)|>l(\e)\}} dt \leq c \int\limits_{\{|x|>\e^{-2}l(\e)\}}\hat{b}(x) dx.\]
  Because $\hat{b} \in L^1(\reals)$ and $\e^{-2}l(\e) \to +\infty$, it follows that $I^\e_1 \to 0$ uniformly in $z\in H_i$\color{black}. The analysis of $I^\e_2$ is analogous to that of $I^\e_1$ and we can conclude that $I^\e_2 \to 0$  unifomly in $z\in H_i$\color{black}.
  To estimate $I^\e_3$, we recall that $\|\varphi\|_{C}^{ 2}\leq c_2$, and therefore a little manipulation shows that
  \[|I^\e_3| \leq \e^{-1} c \|f\|_{C^2}\E_{z}\color{black} \int_0^{\tau^\e} e^{-\lambda t} \hat{\sigma}(\e^{-2}\bar{X}^\e(t)) dt.   \]
  By H\"older's inequality for the integral $\E_{z}\int_0^{\tau^{\e}}e^{-\lambda t}\cdot dt$\color{black},
  \begin{align*}
  &|I^\e_3| \leq \e^{-1} c \|f\|_{C^2} {\left(\E_{z}\color{black} \int_0^{\tau^\e} e^{-\lambda t} dt \right)}^{1/2} {\left(\E_{z}\color{black}\int_0^{\tau^\e} e^{-\lambda t} \hat{\sigma}^2(\e^{-2}\bar{X}^\e(t)) dt \right)}^{1/2}\\
   &\leq c \|f\|_{C^2} {\left( \e^{-2} \E_{z}\color{black}\int_0^{\tau^\e} e^{-\lambda t} \hat{\sigma}^2(\e^{-2}\bar{X}^\e(t)) dt \right)}^{1/2} .
  \end{align*}
  Therefore, $I^\e_3 \to 0$ unifomly in $z\in H_i$ by the same arguments used for $I^\e_1$\color{black}.
  We show that $I^\e_4 \to 0$ by applying Lemma \ref{lem:conv_of fast_motion}.
\end{proof}

\begin{Lemma}\label{lem:longtimeconv4}
  For each $\lambda>0$, and $i\in\{+,-\}$, there is a $u_{i,\lambda}\in D_i$ such that
\[
\frac{1}{2}a_{i}(z)\partial_{xx}u_{i,\lambda}(z)=\lambda u_{i,\lambda}(z)\qquad\forall z\in H_i, |x|<1
\]
and such that $u_{\pm,\lambda}(0,y)=1$ and $\pm\partial_x u_{\pm,\lambda}(0,y)<c<0$.
\end{Lemma}

\begin{proof}
  Set
  \[u_{\pm,\lambda}(x,y) = e^{\mp \sqrt{\frac{2\color{black}\lambda}{a_\pm(y)}}x}.\]
  This function\color{black}~ satisfies  all the requirements\color{black}. 
\end{proof}

\begin{Lemma}\label{lem:longtimeconv3}
 Let $\delta(\e)\to 0$. Then \eqref{eq:Property3} holds for $Z^\e$.
\end{Lemma}

\begin{proof}
   First we observe that
   \[\E_z \int_0^\infty e^{-\lambda t} \chi_{(-\delta(\e),\delta(\e))}(\bar{X}^\e(t)) dt \leq \sum_{k=0}^\infty e^{-\lambda k} \E_z \int_{k}^{k+1} \chi_{(-\delta(\e),\delta(\e))}(\bar{X}^\e(t)) dt. \]
   By Lemma \ref{lem:mylemma}, there exists $c>0$ independent of $\e$ and $k$ such that
   \[\E_z \int_{k}^{k+1} \chi_{(-\delta(\e),\delta(\e))}(\bar{X}^\e(t)) dt \leq c |\chi_{(-\delta(\e),\delta(\e))}|_{L^1} \leq c \delta(\e).\]
   Therefore, for any $z \in \reals^{1+d}$,
   \[\E_z \int_0^\infty e^{-\lambda t} \chi_{(-\delta(\e),\delta(\e))}(\bar{X}^\e(t)) dt \leq \frac{c}{1-e^{-\lambda}} \delta(\e)\to 0\]
as $\e\downarrow 0$.
\end{proof}

\begin{Lemma}\label{lem:longtimeconv2}
  For the system $Z^\e$, $p_+=p_-=\frac{1}{2}$ in the sense that
  \begin{equation}
    \lim_{\e \to 0}\P(Z^\e( \theta^\delta) \in H_+) = \lim_{\e \to 0}\P(Z^\e( \theta^\delta) \in H_-) = \frac{1}{2}
  \end{equation}
  uniformly in $z \in \mathbb{R}^{1+d}$, $|x| < l(\e)$ as long as $l(\e)/\delta(\e)\to 0$.\color{black}
\end{Lemma}

\begin{proof}
Let
  \begin{equation*}
    \P_z(\bar{X}^\e( \theta^\delta) \in H_+) = f^\e(x,y)
  \end{equation*}
Since the exit only depends on the $x$ component, and the effect of $\varphi$ is a time change that does not effect the exit probabilities, it is easy to see that the corresponding Brownian formula holds, i.e.
  \[f^\e(x) = \frac{1}{2 \delta(\e)}(x+\delta(\e)) \]
and if $l(\e)\ll\delta(\e)$,
\[\lim_{\e \to 0} \sup_{|x|\leq l(\e)}\left|\P_{x}(\bar{X}^\e( \theta^\delta) \in H_+) - \frac{1}{2}\right|\to 0,	\]
as $\e\downarrow 0$. The part for $H_-$ follows by complementation.\color{black}
\end{proof}

\begin{Lemma}\label{lem:longtimeconv1}
  For $i,j=1,...,n$, let
  \begin{equation*}
    \beta_i(y) = \int_{-\infty}^\infty \frac{b_i}{|\varphi|^2}(x,y) dx,\qquad \alpha_{ij}(y) = \int_{-\infty}^\infty \frac{(\sigma\sigma^T)_{ij}}{|\varphi|^2}(x,y) dx.
  \end{equation*}
  Then if $\delta(\e),l(\e)\to 0$ such that $\delta(\e)/l(\e)\rightarrow\infty$, and $ \theta^\delta = \inf\{t>0: |X^\e(t)|>\delta\}$,
  \begin{equation*}
    \frac{1}{\delta}\E_z(\bar{Y}^\e_i( \theta^\delta) - y_i) \to \beta_i(y),\qquad     \frac{1}{\delta}\E_z(\bar{Y}^\e_i( \theta^\delta) - y_i)(\bar{Y}^\e_j( \theta^\delta) - y_j) \to \alpha_{ij}(y)
  \end{equation*}
  as $\e \to 0$ for $|x|< l(\e)$. Moreover,
  \begin{equation*}
 \sup_{|x|\leq l(\e)}\E|\bar{Y}^{\e}_i( \theta^\delta)-y|^3=o(\delta).
  \end{equation*}
\end{Lemma}

\begin{proof}
 Observe that
  \begin{equation}\label{eq:lem561}
    \frac{1}{\delta} \E_z(\bar{Y}^\e_i( \theta^\delta) - y_i) = \frac{1}{\delta\e^{2}} \E_z \int_0^{ \theta^\delta} b_i(\e^{-2}\bar{X}^\e(s),\bar{Y}^\e(s)) ds.
  \end{equation}
By Lemma \ref{lem:mylemma} and the Lipschitz continuity of $b$, and the fact that

\[|b_i(\e^{-2}\bar{X}^\e(t), \bar{Y}^\e(t)) - b_i(\e^{-2}\bar{X}(t),y)| \leq 2 \hat{b}(\e^{-2}\bar{X}(t)) \chi_{\{|\e^{-2}\bar{X}^\e(t)| \geq K\}} + c |\bar{Y}^\e(t) - y| \chi_{\{|\bar{X}^\e| \leq \e^{2}K\}}, \]
we have
\begin{align}\label{eq:lem562}
I(\delta,\e)=\frac{1}{\delta\e^2}\sup_{|x|\leq l(\e)}\left|\E_z \int_0^{ \theta^\delta} b_i(\e^{-2}\bar{X}^\e(s),\bar{Y}^\e(s)) ds- \E_z \int_0^{ \theta^\delta} b_i(\e^{-2}\bar{X}^\e(s), y)ds\right|\\
\nonumber\leq  \frac{2}{\delta}\sqrt{\E \theta^\delta}||\mathbbm{1}_{\{x\geq K\}}\hat{b}||_{L^1}+\frac{c}{\delta\e^2}\E_z\int_0^{ \theta^\delta}\mathbbm{1}_{\{|\bar{X}^{\e}(s)|<K\e^2\}}|\bar{Y}^{\e}(s)-y|ds.
\end{align}
Using $\E_z \theta^\delta\leq C\delta^2$ for all $|x|\leq l(\e)$, we see that the first term can be made less than $\eta/2$ for any $\eta>0$ if $K$ is chosen sufficiently large. Fix such a $K$ and note that by Cauchy-Schwarz inequality, the second term in (\ref{eq:lem562}) can be bounded from above by
\begin{align}\label{eq:lem563}
\frac{1}{\delta\e^2}\E_z&\left[\sup_{t\in[0, \theta^\delta]}|\bar{Y}^{\e}(t)-y|\int_0^{ \theta^\delta}\mathbbm{1}_{\{|\bar{X}^{\e}(s)|<K\e^2\}}ds\right]\leq \\
\nonumber&\leq\frac{1}{\delta}\left(\E_z\sup_{t\in[0, \theta^\delta]}|\bar{Y}^{\e}(t)-y|^2\right)^{1/2}\left[\E_z\left(\frac{1}{\e^2}\int_0^{ \theta^\delta}\mathbbm{1}_{\{|\bar{X}^{\e}(s)|<K\e^2\}}ds\right)^2\right]^{1/2}
\end{align}
where the second term is easily seen to be less than $c(\E \theta^\delta)^{1/2}\leq c\delta$ by Lemma \ref{lem:mylemma}.  On the other hand, $(a+b)^2\leq 2(a^2+b^2)$ and the BDG inequality imply that
\begin{align}\label{eq:lem564}
\E_z&\left(\sup_{t\in[0, \theta^\delta]}|\bar{Y}^{\e}(t)-y|^2\right)\leq\\
\nonumber &\leq 2\E\left(\frac{1}{\e^2}\int_0^{\theta}\hat{b}(\e^{-2}\bar{X}^{\e}(s))ds\right)^2+2c\E\left(\frac{1}{\e^2}\int_0^{\theta^{\delta}}\hat{\sigma}(\e^{-1}\bar{X}^{\e}(s))ds\right)\leq c\E_z \theta^\delta\leq c\delta^2,
\end{align}
where Lemma \ref{lem:mylemma} and Lemma \ref{lem:taudeltaestim} were used in the last two inequalities. By (\ref{eq:lem563}), and (\ref{eq:lem564}), the second term in (\ref{eq:lem562}) is less than $c\delta<\eta/2$ and consequently $I(\delta,\e)<\eta$  for sufficiently small $\delta$. Since $\eta$ was arbitrary, (\ref{eq:lem561}) implies that
  \begin{equation*}
    \frac{1}{\delta} \E_z(\bar{Y}^\e_i( \theta^\delta) - y_i) =  \frac{1}{\delta \e^2} \E_z \int_0^{ \theta^\delta} b_i(\e^{-2}\bar{X}^\e(s), y)ds+o(1).
  \end{equation*}
  By the properties of the local time, the above expression is equal to
  \begin{equation*} \label{eq:martingale-meyer-tanaka}
    \frac{1}{\e^2\delta} \E_z \int_{-\infty}^\infty \frac{b_i}{|\varphi|^2}(\e^{-2} u, y) L^{\bar{X}^\e}(u, \theta^\delta)du = \frac{1}{\delta} \E_z \int_{-\infty}^\infty \frac{b_i}{|\varphi|^2}(u,y) L^{\bar{X}^\e}(\e^2 u,  \theta^\delta) du
  \end{equation*}
  where $L^{\bar{X}^\e}(u,t)$ is the local time of $\bar{X}^\e$ at $u$.
  We calculate  by the Tanaka formula\color{black}~ that for fixed $x$,
  \begin{equation*}
    \frac{1}{\delta}\E_z L^{\bar{X}^\e}(\e^2 u,  \theta^\delta) = \frac{1}{\delta}\E \left(|\bar{X}^\e( \theta^\delta) - \e^2 u| - | x - \e^2 u| \right)
  \end{equation*}
  Because $\bar{X}^\e( \theta^\delta) = \pm \delta$, the above expression is bounded by
  \[\frac{\delta + |x|}{\delta}\]
  and if $|x|\leq l(\e) \ll\delta(\e)$,
  \[\lim_{\e \to 0} \frac{1}{\delta} \E_z L^{\bar{X}^\e}(\e u, \theta^\delta) = 1.\]
  Therefore, it follows from \eqref{eq:martingale-meyer-tanaka} and the dominated convergence theorem that
  \[\frac{1}{\delta} \E_z(\bar{Y}^\e_i( \theta^\delta) - y_i)  \to \int_{-\infty}^\infty \frac{b_i}{|\varphi|^2\color{black}}(x,y) dx.\]

  Similar calculations show that
  \[\frac{1}{\delta} \E_z(\bar{Y}^\e_i( \theta^\delta)-y_i)(\bar{Y}^\e_j( \theta^\delta)-y_j)=\frac{1}{\delta\e^2}\E_z\int_0^{ \theta^\delta}(\sigma\sigma^T)_{ij}(\e^{-2}\bar{X}^{\e}(t),\bar{Y}^{\e}(t))dt\to \int_{-\infty}^\infty \frac{(\sigma \sigma^T)_{ij}}{|\varphi|^2}(x,y) dx.\]
and a further straightforward application of Lemma \ref{lem:mylemma} yields
\[
 \sup_{|x|\leq l(\e)}\E|\bar{Y}^{\e}_i( \theta^\delta)-y|^3=\mathcal{O}(\delta^{3/2})=o(\delta).
\]
\end{proof}
}

Set $\gamma>0$, $l(\e)=\e^{2-\gamma}$ and $\delta(\e)=\max\{\e^{2-2\gamma},\e^{-\gamma}\sup_{|x|>l(\e)}| k(\e)|\}$  where $k(\e)$ is the convergence rate in Lemma \ref{lem:longtimeconv5}. Lemmas \ref{lem:longtimeconv5}-\ref{lem:longtimeconv1} and Theorem \ref{thm:F-W-1993} together with tightness imply that every subsequence of $(\bar{X}^\e,\bar{Y}^{\e})$ has a further subsequence converging to a solution of the martingale problem associated to $\bar{\mathcal{L}}$. In the next section we will argue that this martingale problem is well posed which establishes the convergence of the entire sequence.

%%%%%%%%%%%%%%%%%%%%%%% NORMALIZED DEVIATION %%%%%%%%%%%%%%%%%%%

\subsection{Convergence to the solution of the martingale problem in the normalized deviation case}
In this section, we show that the assumptions of Theorem \ref{thm:F-W-1993} are satisfied by $Z^\e=(X^\e,y,\zeta^\e)$ with $Z^{\e}(0)=z=(x,y,\xi)$ under the assumptions of Theorem \ref{main_result2} or Theorem \ref{main_result1}. We start with a simple growth estimate on $\zeta^{\e}$.

\begin{Lemma}
For any stopping time $\tau>0$ with $\E\tau<\infty$,
\begin{enumerate}
\item[(1)] under the conditions of Theorem \ref{main_result1}, we have
\begin{equation}\label{eq:supzeta}
\E_z\sup_{0\leq t\leq \tau}|\zeta^{\e}(t)|^2\leq C\left( \E_ze^{ 4||b_1||_{C^1}\tau}\right)^{1/2}(|\xi|^4+\E_z\tau^2)^{1/2},
\end{equation}
\item[(2)] under the conditions of Theorem \ref{main_result2}, we have
\begin{equation}\label{eq:supzeta2}
\E_z\sup_{0\leq t\leq \tau}|\zeta^{\e}(t)|^2\leq \left( \E_ze^{ 4||b_1||_{C^1}\tau}\right)^{1/2}(|\xi|^4+\E_z\tau+\e^2\E_z\tau^2)^{1/2}.
\end{equation}
\end{enumerate}
\end{Lemma}

\begin{proof}
Recall that under the conditions of Theorem \ref{main_result1}, we have
\[
d\zeta^{\e}(t)=\frac{1}{\e}[b_1(y(t)+\e\zeta^{\e}(t))-b_1(y(t))]+\frac{1}{\e}b_2\left(\e^{-1}X^{\e}(t),y(t)+\e\zeta^{\e}(t)\right)dt,
\]
which implies
\[
|\zeta^{\e}(t)|\leq |\xi|+||b_1||_{C^1}\int_0^t|\zeta^{\e}(s)|ds+\int_0^t\e^{-1}\hat{b}(\e^{-1}X^{\e}(s))ds.
\]
By Gronwall's lemma,
\[
\sup_{0\leq t\leq \tau}|\zeta^{\e}(t)|\leq\left( |\xi|+\int_0^{\tau}\e^{-1}\hat{b}(\e^{-1}X^\e(s))ds\right)e^{||b_1||_{C^1}\tau},
\]
and therefore by the Cauchy-Schwartz inequality,
\[
\E_z\sup_{t\in[0,\tau]}|\zeta^{\e}(t)|^2\leq C\left( \E_ze^{ 4||b_1||_{C^1}\tau}\right)^{1/2}\left(|\xi|^{ 4}+\E_z\left(\int_0^{\tau}\e^{-1}\hat{b}(\e^{-1}X^{\e}(t))dt\right)^{ 4}\right)^{1/2},
\]
and Lemma \ref{lem:mylemma} and Lemma \ref{lem:taudeltaestim} yield (1).

To prove (2), note that under the conditions of Theorem \ref{main_result2},
\begin{align*}
d\zeta^{\e}(t)=\frac{1}{\sqrt{\e}}[b_1(y(t)+\sqrt{\e}\zeta^{\e}(t))-b_1(y(t))]+\frac{1}{\sqrt{\e}}b_2\left(\e^{-1}X^{\e}(t),y(t)+\sqrt{\e}\zeta^{\e}(t)\right)dt+\\
+\frac{1}{\sqrt{\e}}\sigma\left(\e^{-1}X^{\e}(t),y(t)+\sqrt{\e}\zeta^{\e}(t)\right)dW(t)
\end{align*}
and one can easily see as above that
\begin{align*}
\E_z\sup_{t\in[0,\tau]}&|\zeta^{\e}(t)|^2\leq C\left( \E_ze^{ 4||b_1||_{C^1}\tau}\right)^{1/2}\cdot\\
&\cdot\left(|\xi|^{ 4}+\e^2\E_z\left(\int_0^{\tau}\e^{-1}\hat{b}(\e^{-1}X^{\e}(t))dt\right)^{ 4}+\E_z\left(\int_0^\tau\e^{-1}\hat{\sigma}(\e^{-1}X^{\e}(s))ds\right)^2\right)^{1/2},
\end{align*}
where we used the BDG inequality in the last term. (2) once again follows by applying Lemma \ref{lem:mylemma} twice and Lemma \ref{lem:taudeltaestim}.
\end{proof}

\begin{Lemma}\label{lem:devcase1}
There is a $\lambda_0>0$ such that for $\lambda>\lambda_0$, if we choose $l(\e)$ so that $l(\e) \to 0$ and $\e^{-1} l(\e) \to +\infty$, then \eqref{eq:Property1} is satisfied with some rate $k(\e)$ for $Z^\e(t)= (X^\e(t),y(t), \zeta^\e(t))$ with
  \[\bar{\mathcal{L}} f(x,y,\xi) = \frac{1}{2}a_{\pm}(x,y) f_{xx}(x,y,\xi) + b_1(y)\cdot f_y(x,y,\xi) + (\partial_y b_1(y) \cdot \xi)\cdot f_\xi(x,y,\xi).\]
\end{Lemma}

\begin{proof}
 We only prove the lemma under the hypothesis of Theorem \ref{main_result2}, the other case being almost identical. Using Ito's formula,
\[
\E_z\left[e^{-\lambda\tau^{\e}}f(Z^{\e}(\tau^{\e}))-f(z)+\int_0^{\tau^{\e}}e^{-\lambda t}(\lambda f(Z^{\e}(t))-\bar{\mathcal{L}}_if(Z^{\e}(t)))dt\right]=I_1(z,\e)+I_2(z,\e)
\]
where $I_1(z,\e)$ contains terms that can be dealt with exactly as we did with their counterparts in Lemma \ref{lem:longtimeconv5} and therefore
\[
\sup_{ z\in H_i }|I_1(z,\e)|\rightarrow 0.
\]
On the other hand,
\[
I_2(z,\e)=\frac{1}{2}\E_z\int_0^{\tau^{\e}}e^{-\lambda t}f_{xx}(X^{\e}(t),y(t),\zeta^{\e}(t))\left[\frac{b_1(y(t) +\sqrt{\e}\zeta^\e(t))-b_1(y(t))}{\sqrt{\e}}-(\partial_yb_1)(y(t))\zeta^{\e}(t)\right]dt,
\]
and because $b_1(y)\in C_{ b}^{2}(\mathbb{R}^{d})$, we have that
\begin{equation}\label{eq:iketto}
|I_2(z,\e)|\leq C\sqrt{\e}\int_0^{\infty}e^{-\lambda t}\E_z|\zeta^{\e}(t)|^2dt,
\end{equation}
where the constant $C$ depeds on $|f|_{C^2}$ and the second derivatives of $b_1$.  By (\ref{eq:supzeta2}), we have for all $t>0$
\[
\E_z|\zeta^{\e}(t)|^2\leq C_1te^{C_2t}
\]
for some $C_1,C_2>0$ and therefore choosing $\lambda_0=C_2$ makes the integral in (\ref{eq:iketto}) converge for every $\lambda>\lambda_0$ and the result is proved.
\end{proof}

\begin{Lemma} \label{lem:eigenfunction-deviation-case}
  For each $\lambda>0$, and $i=+,-$, there is a $u_{i,\lambda} \in D_i$ such that
  \[\bar{\mathcal{L}}_i u_{i,\lambda} = \lambda u_{i,\lambda}.\]
\end{Lemma}

\begin{proof}
  The operators $\mathcal{L}_\pm$ are both generators of non-degenerate diffusion processes $Z=(X,y,\zeta)$ in $H_\pm$. Define
  \[\tau=\inf\{t>0: X(t)=0\},\]
  then it is well known that
  \begin{equation}
    u_{i,\lambda}(z) := \E_{z} e^{-\lambda \tau}
  \end{equation}
  has the property we desire. That the requirement on the derivative is also fulfilled once again follows from a time change argument and the corresponding Brownian formula.
\end{proof}

The statements of the following lemma are proved completely analogously to those in the previous section.
\begin{Lemma}
  Property \eqref{eq:Property3} holds for $Z^\e$. Also, $p_+=p_-=\frac{1}{2}$ in the sense that
  \[\lim_{\e \to 0} \P_z \left( Z^\e(\tau^\e) \in H_+\right) = \lim_{\e \to 0} \P_z \left( Z^\e(\tau^\e) \in H_-\right) = \frac{1}{2}\]
  uniformly for $(x,y,z) \in \reals^{1 + 2d}$, $|x|<l(\e)$.
\end{Lemma}

\begin{Lemma}\label{lem:devcase5}
  For $i,j=1,..,n$, let
  \begin{equation*}
   \beta_i(y) = \begin{cases}\int_{-\infty}^\infty \frac{b_{2,i}}{|\varphi|^2}(x,y) dx & \text{ if } \sigma \equiv 0,
   \\ 0 & \text{ otherwise,}
   \end{cases}
\qquad
    \alpha_{ij}(y) = \int_{-\infty}^\infty \frac{(\sigma\sigma^T)_{ij}}{|\varphi|^2}(x,y) dx.
  \end{equation*}
  Then if $\delta(\e)/l(\e)\to\infty$,
  \begin{equation}\label{eq:normdevgluing}
    \frac{1}{\delta} \E_z(\zeta^\e_i({ \theta^\delta}) -\xi_i\color{black}) \to \beta_i(y),\qquad
    \frac{1}{\delta} \E_z (\zeta^\e_i({ \theta^\delta}) - \xi_i\color{black})(\zeta^\e_j({ \theta^\delta}) - \xi_j\color{black}) \to \alpha_{ij}(y)
  \end{equation}
 as $\e \to 0$ uniformly in $|x|<l(\e)$. Moreover,
\begin{equation}\label{eq:normdevthirdterm}
\E_z|\zeta_i^{\e}({ \theta^\delta})-\xi|^3=o(\delta).
\end{equation}
\end{Lemma}

\begin{proof}
We only prove the case $\sigma\equiv 0$, the other one being similar. First note that by (\ref{eq:supzeta}) and Lemma \ref{lem:taudeltaestim}, there is a $\delta_0>0$ and $M>0$ such that
\begin{equation}\label{eq:supzetabounded}
\E_z\sup_{t\in[0,{ \theta^\delta}]}|\zeta^{\e}(t)|^2\leq M,
\end{equation}
whenever $\delta<\delta_0$.
Now write
\begin{align}\label{eq:bdrdrift}
 \frac{1}{\delta} &\E_z(\zeta^\e_i({ \theta^\delta}) -\xi_i)=\\
\nonumber&=\frac{1}{\delta\e}\E_z\int_0^{{ \theta^\delta}}\left[b_1(y(s)+\e\zeta^{\e}(u))-b_1(y(s))\right]dt+\frac{1}{\delta\e}\E_z\int_0^{{ \theta^\delta}}b_2(\e^{-1}X^{\e}(s),y(s)+\e\zeta^{\e}(s))ds
\end{align}
The second term can be dealt with exactly as in (\ref{eq:lem562}) and (\ref{eq:lem563}) and it can be shown to converge to
\[
\beta(y)=\int_{-\infty}^{\infty}\frac{b_2}{|\varphi|^2}(x,y)dx
\]
uniformly in $|x|<l(\e)$ once one notices that
\[
\E_z\sup_{t\in[0,{ \theta^\delta}]}|y(t)+\e\zeta^{\e}(t)-y|^2\leq C\E_z\sup_{t\in[0,{ \theta^\delta}]}|y(t)-y|^2+\e^2\E_z\sup_{t\in[0,{ \theta^\delta}]}|\zeta^{\e}(t)|^2=o(1)
\]
by (\ref{eq:supzetabounded}), a standard Gronwall estimate on $y(t)$, and Lemma \ref{lem:taudeltaestim}.
On the other hand, the integral in the first term in (\ref{eq:bdrdrift}) is immediately less than or equal to
\[
\frac{||b_1||_{C^1}}{\delta}\E_z\int_0^{{ \theta^\delta}}|\zeta^{\e}(t)|dt\leq\frac{||b_1||}{\delta}\E_z\left[{ \theta^\delta}\cdot\sup_{t\in[0,{ \theta^\delta}]}|\zeta^{\e}(t)|\right]\leq\left(\E_z\sup_{t\in[0,{ \theta^\delta}]}|\zeta^{\e}(t)|^2\right)^{1/2}\sqrt{\delta^{-2}\E_z({ \theta^\delta})^2}
\]
which converges to zero by (\ref{eq:supzeta}) and Lemma \ref{lem:taudeltaestim}.

The second claim in (\ref{eq:normdevgluing}) and (\ref{eq:normdevthirdterm}) can be proved similarly.
\re
\end{proof}

If we set $\gamma>0$, $l(\e)=\e^{1-\gamma}$ and $\delta(\e)=\max\{\e^{1-2\gamma},\e^{-\gamma}k(\e)\}$ then Lemmas \ref{lem:devcase1}-\ref{lem:devcase5} and Theorem \ref{thm:F-W-1993} with $(y,\xi)$ in place of $y$ show that the possible subsequential limits are all solutions of the martingale problem associated to $\bar{\mathcal{L}}$ and the convergence will be established once we show well-posedness in the next section.

%%%%%%%%%%%%%%%%%%%%%%%%%%%%%%%%%%%%%%%%%%%%%%%%%%%%%%%%%%
%%%%%%%%%%%%%%%%%%%%%%%%% SECTION 6 %%%%%%%%%%%%%%%%%%%%%%%%%%%
%%%%%%%%%%%%%%%%%%%%%%%%%%%%%%%%%%%%%%%%%%%%%%%%%%%%%%%%%

\section{Uniqueness and Markov property of the martingale problem and characterization of the solution.}\label{sec:uniqmarting}

In this section, we prove that the martingale problems associated to the operators in Theorem \ref{main_result2}, Theorem \ref{main_result1} and Theorem \ref{main_result3} are well-posed and the unique solutions satisfy the stated SDEs.

Our main tool will be the following result which is a special case of a result of Ethier and Kurtz which we reformulate for our current purposes.

 \begin{Theorem}[{\cite[Theorem 4.4.1]{EK}} Uniqueness of the Martingale Problem] \label{EK-uniqueness}
   Let $\mathcal{L}$ be linear and dissipative on $C_0(\mathbb{R}^{d+1})$. If $\overline{D(\mathcal{L})} = C_0(\mathbb{R}^{d+1})$ and there exists a $\lambda>0$ such that $\overline{(\lambda - \mathcal{L})(D(\mathcal{L}))} = C_0(\mathbb{R}^{d+1})$, then the martingale problem is unique for $\mathcal{L}$. That is, if $Z, \tilde{Z}$ are processes in $C([0,\infty),\mathbb{R}^{1+d})$ and $Z(0) = \tilde{Z}(0)$ in distribution and for any $f \in D$,
   \[f(Z(t)) - \int_0^t Lf(Z(s))ds \text{ and } f(\tilde{Z}(t)) - \int_0^t Lf(\tilde{Z}(s))ds,\]
  are martingales with respect to the filtration generated by the coordinate process then $Z$ and $\tilde{Z}$ are equal in distribution. Moreover, if there is a solution $Z$, it is a Markov process corresponding to the semigroup on $C_0(\mathbb{R}^{d+1})$ generated by the closure of $L$.
 \end{Theorem}

In  all three \re cases, we probabilistically construct a solution of the corresponding stochastic differential equation and conclude by the Hille-Yosida theorem that the generator of the associated Markov semigroup satisfies the hypothesis of Theorem \ref{EK-uniqueness}. This together with  Lemma \ref{lem:itm} implies that the measure generated by this process is the unique solution of the martingale problem associated to $\bar{\mathcal{L}}$.
\subsection{Long-time case}
{
By the concrete form of $\beta$ and $\alpha$ (as in Theorem \ref{main_result3}), there is a weak solution $(Y_0,W_2)$ to the equation
\[
dY_0(t)=\beta(Y_0(t))dt+\sqrt{\alpha(Y_0(t))}dW_2(t).
\]
Take an independent Brownian moton $W_1$ and consider the time changed process
\[
\hat{Z}(t) = (\hat{X}(t), \hat{Y}(t))=(W_1(t),Y_0(L^{W_1}(t,0)))
\]
where $L^{W_1}(t,0)$ is the local time of $W_1$ at zero. Finally, set
\[
t(s)=\int_0^s\frac{1}{a_{\pm}(W_1(u),\hat{Y}(u))}du
\]
and let $s(t)$ denote its inverse. Let $Z(t)=\hat{Z}(s(t))$.

\begin{Lemma}
$Z(t)=:(\bar{X}^0(t),\bar{Y}^{0}(t))$ is a weak solution of (\ref{eq:longlimX})-(\ref{eq:longlimY}).
\end{Lemma}

\begin{proof}
First note that if we introduce $V^{W_1}(t)=W_2(L^{W_1}(t,0))$ then $\hat{Y}$ satisfies
\[
\hat{Y}(t)=\int_0^t\beta(\hat{Y}(s))L^{W_1}(ds,0)+\int_0^t\sqrt{\alpha(\hat{Y}(s))}dV^{W_1}(t)
\]
Note that $V^{W_1}$ is a continuous martingale and
\[
\P\left(\sum_{i,j=1}^d \int_0^t\alpha_{ij}(\hat{Y}(s))d\langle V_j^{ W_1},V_i^{ W_1}\rangle(s)<\infty\right)=1
\]
for all $t>0$ and therefore the stochastic integral is well defined for all times. { Indeed, it follows from the construction that $\left<V_i^{ W_1}, V_j^{ W_1}\right> = L^{W_1}(t,0) \delta_{ij}$. We also calculate that by \eqref{eq:decay_cond}
  \[\sum_{i=1}^d \alpha_{ii}(y) = \int_{-\infty}^\infty \frac{\Tr(\sigma\sigma^T)}{|\varphi|^2}(x,y) dx \leq c|\hat{\sigma}|_{L^2(\reals)}^2.\]
  Therefore,
  \[\sum_{i,j=1}^d \int_0^t \alpha_{ij}(\hat{Y}(s)) d \left<V_j^{ W_1},V_i^{ W_1}\right>(s) \leq c|\hat{\sigma}|_{L^2(\reals)}^2 L^{W_1}(t,0).\]}
 which is finite with probability 1. Now performing the second time change yields
\begin{align*}
\bar{X}^0(t)&=W_1(s(t))\\
\bar{Y}^0(t)&=\int_0^t\beta(\bar{Y}^0(u))L^{W_1}(s(du),0)+\int_0^t\alpha(\bar{Y}^0(u))V^{W_1}(s(du))
\end{align*}
Note that for any Borel measurable $f:\mathbb{R}\to[0,\infty)$,
\begin{align*}
\int_{\mathbb{R}}f(x)L^{W_1}(s(t),x)dx=\int_0^{s(t)}f(W_1(u))du=\int_0^tf(W_1(s(u)))a_{\pm}(W_1(s(u)),\hat{Y}(s(u)))du=\\
=\int_0^tf(\bar{X}^0(u))a_{\pm}(\bar{X}^0(u ),\bar{Y}^0(u))du=\int_{\mathbb{R}}f(x)L^{\bar{X}^0}(t,x)dx,
\end{align*}
which implies $L^{W_1}(s(t),x)=L^{\bar{X}^0}(t,x)$. This also implies that if we let $V^{\bar{X}^0}(t)=V^{W_1}(s(t))$, then we have
\[
\langle V^{\bar{X}^0}_i, V^{\bar{X}^0}_j\rangle(t)=\langle V^{W_1}_i, V^{W_1}_j\rangle(s(t))=\delta_{ij}L^{W_1}(s(t),x)=\delta_{ij}L^{\bar{X}^0}(t,x).
\]
This means that
\[
\bar{Y}^0(t)=\int_0^t\beta(\bar{Y}^0(u))L^{\bar{X}^0}(du,0)+\int_0^t\alpha(\bar{Y}^0(u))dV^{\bar{X}^0}(u)
\]
as required. We also have that the martingale
\[
W(t):=\int_0^t\frac{1}{\sqrt{a_{\pm}(\bar{X}^0(u),\bar{Y}^0(u))}}W_1(s(du))
\]
has quadratic variation $t$ and therefore by Levy's theorem it is a Brownian motion. Consequently,
\[
\bar{X}^0(t)=\int_0^t\sqrt{a_{\pm}(\bar{X}^0(u),\bar{Y}^0(u))}dW(u).
\]

It remains to show that the appropriate cross-variations vanish. Since $\langle W_1, W_2\rangle\equiv 0$, it is easy to see that $\langle W_1,V_i^{W_1}\rangle_t\equiv 0$. As $s'(t)\in [c_1,c_2]$, $\langle \bar{X}^0, V_i^{\bar{X}^0}\rangle_t=\langle W_1,V_i^{W_1}\rangle_{s(t)}\equiv 0$.  Moreover,
\[
\langle W,V_i^{\bar{X}^0}\rangle_t=\int_0^t\frac{1}{\sqrt{a_{\pm}(\bar{X}^0(u),\bar{Y}^0(u))}}d\langle \bar{X}^0,V_i^{\bar{X}^0}\rangle_u\equiv 0.
\]
\end{proof}

\noindent\textit{Proof of Theorem \ref{main_result3}.} As a Markov process, $Z(t)$ generates a strongly continuous contraction semigroup $T_t$ on $C_0(\mathbb{R}^{1+d})$ by $T_tf(z)=\E_zf(Z(t))$. From this definition and Lemma \ref{lem:itm}, it is easy to see that the closure of $\bar{\mathcal{L}}$ (denoted by $\bar{\mathcal{L}}$ again) is an extension of the generator of $T_t$ and hence, by a standard semigroup fact, they are actually equal. Thus, by the Hille Yosida theorem, $\mathcal{D}(\bar{\mathcal{L}})$ is dense in $C_0(\mathbb{R}^{1+d})$ and $\mathcal{R}(\lambda-\bar{\mathcal{L}})=C_0(\mathbb{R}^{1+d})$. Therefore by Theorem \ref{EK-uniqueness}, the associated martingale problem is well-posed. It is straightforward to show using Lemma \ref{lem:itm}, that $Z(t)$ is actually a solution of the martingale problem. This, combined with the tightness result of Section 4 and the convergence result of Section 5, finishes the proof of Theorem \ref{main_result3}.\qed

\subsection{Normalized deviation case}
Similarly to the previous case, we are going to directly construct the limiting Markov process which solves the martingale problem associated to $\mathcal{L}$. We are only going to carry out the construction, the other steps are analogous.
\begin{Lemma}
There is a process $X^0$ which solves
  \[dX^0(t) = \sqrt{a_\pm(X^0(t),y(t))}dW(t).\]
  This $X^0$ solves the martingale problem associated to $\bar{\mathcal{L}}$ for functions $f$ that do not depend on $y$ or $\zeta$.
\end{Lemma}
\begin{proof}
  Such a result is not completely obvious because of the discontinuity of the coefficients. Consider instead the system
  \begin{align*}
    d\hat{X}(s) = dW(s)\\
    d\hat{y}(s) = \frac{b_1(\hat{y}(s))}{a_\pm(\hat{X}(s), \hat{y}(s))}ds.
  \end{align*}
  The pair $(\hat{X},\hat{y})$ is perfectly well-defined because the formula for $\hat{y}$ is a random ODE that we can consider pathwise in the integral sense. The function $b_1(y)/a_\pm(x,y)$ is discontinuous in $x$ , but uniformly Lipschitz continuous in $y$. Classical successive approximation arguments show that a unique solution exists.
  We then define the random time change
  \[t(s) = \int_0^s \frac{1}{a_\pm(\hat{X}(r), \hat{y}(r))} dr\]
  and denote its functional inverse by $s(t)$.
  Then we define
  \begin{align*}
    &X^0(t) = \hat{X}(s(t)),\\
    &y(t) = \hat{y}(s(t))
  \end{align*}
 which has the properties that we desire.
\end{proof}

The following two lemmas are a straightforward consequence of Lemma \ref{lem:itm}.

\begin{Lemma} \label{thm:existence-2}
  ~\\ \begin{enumerate}
\item Let $(X^\e,y,\zeta^\e)$ satisfy the assumptions of Theorem \ref{main_result2}. Let
  \begin{equation}
    \zeta^0(t) = \int_0^t e^{\int_s^t \partial_y b_1(y(r))dr}\sqrt{\left(\int_{-\infty}^\infty (\sigma\sigma^T)(x,y(s))dx\right)}dV^{X^0}(s).
  \end{equation}
 where $V^{X^0}=W_0(L^{X^0}(t,0))$\re. Then the triple $(X^0,y,\zeta^0)$ is a Markov process with generator $\bar{\mathcal{L}}$.

 \item Let $(X^\e,y,\zeta^\e)$ satisfy the assumptions of Theorem \ref{main_result1}. Let
  \begin{equation}
    \zeta^0(t) = \int_0^t e^{\int_s^t\partial_y b_1(y(r))dr} \sqrt{\left(\int_{-\infty}^\infty b(x,y(s)) dx \right)} b(X^0(s),y(s)) L^{X^0}(ds,0)
  \end{equation}
  where the above integral is the Riemann-Stieltjes integral with respect to the increasing function $t \mapsto L^{X^0}(t,0)$.
  Then the triple $(X^0,y,\zeta^0)$ is a Markov process with generator $\bar{\mathcal{L}}$.
\end{enumerate}
\end{Lemma}

\noindent\textit{Proof of Theorem \ref{main_result2} and Theorem \ref{main_result1}.}
Analogous to the proof of Theorem \ref{main_result3} in the previous subsection.
\qed

\section{Remarks and further directions}\label{rem:further}
There are several interesting further questions one can ask in relation with (\ref{eq:X_eq})-(\ref{eq:Y_eq}).
\begin{itemize}
\item In some sense, the original result of \cite{KK04} and the results in this paper carry out the first few steps of the program contained in (\cite{FW12}, Chapter 7) when the fast process is positively recurrent. In particular, a result similar to Theorem \ref{main_result2} was proved by Khasminskii (\cite{K66}) and one simlar to Theorem 5 was proven again by Borodin (\cite{B77}) sharpening a result of Khasminskii (\cite{K66b}) that in turn was inspired by the non-rigorous result of Stratonovich (\cite{S66}). The next step in this program would be to study large deviation estimates of $(X^{\e},Y^{\e})$ from their respective limit in \cite{KK04}.

\item Along the same lines, it is expected that one can generalize Theorem \ref{main_result3} to the case when there are several conserved quantities of that are conserved by the limiting motion in \cite{KK04}. Namely, if $H$ is such a quantity and $(X^0,Y^0)$ is the limit of $(X^{\e},Y^\e)$ in distribution then $H(Y^{\e}(t))$ will converge to the constant process process $H(y)$. uniformly on compact time intervals in probability. It is natural to conjecture that this process needs to be considered on time-scales of order $\e^{-2}$ in order to see any non-trivial behavior. The corresponding result for positive recurrent averaging was carried out by Borodin and Freidlin (\cite{BF95}). The result is expected to be a combination of external averaging due to the null recurrent fast process and internal averaging inside the over the level sets of $H$. In the case of Theorem \ref{main_result3}, $H(y)=y$ and the level sets are single points, so this internal averaging does not take place.

\item A different direction is to replace the fast motion with a more general null recurrent process or a process that is only neighborhood recurrent. For example, by considering a situation when the fast process is driven by a Bessel process of order $n\in (1,2]$, one can hope to study the asymptotic behavior of a three dimensional diffusive stochastic dynamical system the dynamics of which is perturbed in a narrow tube.

\item The results in this paper also imply certain results on partial differential equations through the well known representation formulas. For example, consider the Cauchy problem for $u^{\e}:\mathbb{R}^{1+d}\times \mathbb{R}_+\to\mathbb{R}$
\begin{align*}
\partial_t u^{\e}&=\frac{1}{2}\left|\varphi\left(\e^{-1}{x},y\right)\right|^2\partial_{xx}u^{\e}+\left[b_1(y)+b_2\left(\e^{-1}{x},y\right)\right]\partial_yu^{\e}+\\
&+\frac{1}{2}\sum_{i,j=1}^d(\sigma\sigma^T)_{ij}\left(\e^{-1}{x},y\right) \partial_{y_iy_j}u^{\e}+\sum_{i=1}^d\sum_{j=1}^k\sigma_{ij}(\e^{-1}x,y)\varphi_j(\e^{-1}x,y)\partial_{xy_{i}}u^{\e}
\end{align*}
with initial condition $u^{\e}(x,y,0)=f_{\e}(x,y)$. It is well known that
\[
u^{\e}(x,y,t)=\E_{(x,y)}f_\e(X^{\e}(t),Y^{\e}(t))
\]
is the unique solution of this Cauchy problem. In this setting, if $f_{\e}(x,y)= f(y)$   then we have by Lemma \ref{lem:Y_close_to_y} that $u^{\e}$ converges to $u^0(x,y,t)=f(y(t))$ uniformly on compact sets which solves the transport equation
\[
\partial_t u^0=b_1(y)\partial_yu^0,\qquad u^0(x,y,t)= f(y).
\]

On the other hand if $b_1\equiv 0$, $f_{\e}(x,y)=f(\e x,y)$, then Theorem \ref{main_result3} implies that $u^{\e}(x,y,t/\e^2)$ converges to $\hat{u}^0(x,y,t)=\E_{(x,y)}f(\bar{X}^0,\bar{Y}^0)$. Formally, this function solves
\begin{equation}\label{eq:deltaequation}
\partial_t\hat{u}^0=\frac{1}{2}a_{\pm}(x,y)\partial_{xx}\hat{u}^{0}+\delta(x)\left[\beta(y)\partial_y\hat{u}^0+\frac{1}{2}\alpha(y)\partial_{yy}\hat{u}^0\right],\qquad \hat{u}^0(0,x,y)=f(x,y),
\end{equation}
where $\delta(x)$ is the Dirac-delta distribution at x=0, even though the solution theory of such an equation is non-trivial. The case when there is only a generalized drift, i.e. $\alpha\equiv 0$, has been studied (see e.g. \cite{P79}, but see also \cite{BE12}) but to our knowledge, there have been no successful attempts to make sense directly of (\ref{eq:deltaequation}) in the general case. Another approach is to note that $Y(t)$ is a fractional diffusion. Indeed, in the case of e.g.  $a_{\pm}\equiv 1$, and $f(x,y)\equiv f(y)$, we have (see \cite{SZ97} or the more general \cite{BM01})
\begin{equation}\label{eq:fractional_eq}
\partial_t^{1/2}\hat{u}^0=\beta(y)\partial_y\hat{u}^0+\frac{1}{2}\alpha(y)\partial_{yy}\hat{u}^0,\qquad\hat{u}^0(0,y)=f(y),
\end{equation}
where $\partial^{\beta}_y$ is the Caputo-derivative, i.e.
\[
\partial_t^{\beta}f(t)=\frac{1}{\Gamma(1-\beta)}\int_0^{\infty}f'(t-r)r^{-\beta}dr
\]
 This latter non-local equation expresses the averaged (non-Markovian) dynamics of the $Y$-process (which is trivial when $a_{\pm}\equiv 1$ and it is expected that a similar equation can be written down for the general case. It would also be interesting to investiate the relation between (\ref{eq:deltaequation}) and (\ref{eq:fractional_eq}).

Similarly, the results in this paper can be also used to study other initial-boundary problems through the usual representation formulas.
\end{itemize}

%%%%%%%%%%%%%%%%%%%%%%%%%%%%%%%%%%%%%%%%%%%%%%%%%%%
%%%%%%%%%%%%%%%%%%%%%%%%%%%%%%%%%%%%%%%%%%%%%%%%%%%
%%%%%%%%%%%%%%%%%%%%%%%%%%%%%%%%%%%%%%%%%%%%%%%%%%%

\begin{appendix}
\section{Martingale problem - SDE equivalence in the general case}
 In this section we establish that there is a one to one correspondence between the solutions of the martingale problem associated to the operator $\bar{\mathcal{L}}$ and the weak solutions of the corresponding SDEs under very general circumstances. Even though this result is not needed for the specific results of this paper, we believe it is of independent interest and useful for future studies.  For convenience, we prove the case $p_+=p_-=1/2$ case rigorously and then outline the necessary changes for the case when $p_+$ and $p_-$ are constants but not necessarily equal.

\begin{Theorem}
Let $\bar{\mathcal{L}}$ be the operator on $C_0(\mathbb{R}^{1+d})$ described in Section 2 with $p_+=p_-=1/2$. Then the distribution of a process $Z(t)$ solves the martingale problem associated to $\bar{\mathcal{L}}$ if and only if  there is a Brownian motion $W$ and a martingale $V^{Z_1}$ such that $(Z(t), W, V^{Z_1})$ is a weak solution of (\ref{eq:SDE_X})-(\ref{eq:SDE_Y}).
\end{Theorem}

\begin{proof}
Using Lemma \ref{lem:itm}, it is straightforward to see that for every weak solution, the distribution of $Z$ solves the martingale problem associated to $\bar{\mathcal{L}}$.

Now we prove the opposite direction. Assume that $(X,Y)$ is a Martingale solution to the operator $\bar{\mathcal{L}}$ given by
 \begin{equation*}
   \bar{\mathcal{L}} f(x,y) = \frac{1}{2} A(x,y) f_{xx}(x,y) + \sum_{k=1}^d B_k(x,y) f_{y_k}(x,y)
 \end{equation*}
 whose domain consists of sufficiently regular functions subjected to the boundary conditions
 \begin{equation}\label{eq:appendix-bdry-conds}
   \frac{1}{2} f_x(0+,y) - \frac{1}{2} f_x(0-,y) + \sum_{k=1}^d \beta_k(y) f_{y_i}(0,y) +\frac{1}{2}\sum_{k,j=1}^d \alpha_{kj}(y) f_{y_ky_j}(y) = 0.
 \end{equation}
Let $\hat{\mathcal{D}}(\bar{\mathcal{L}})$ be defined the same as $\mathcal{D}(\bar{\mathcal{L}})$ but without requiring its members to decay at infinity. This means that it consists of those continuous functions twice continuously differentiable in $x$ and $y$ with the $x$ derivative possibly being discontinuous at $x=0$. Since
\[
f(X(t),Y(t))-f(X(0))-\int_0^t\bar{\mathcal{L}}f(X(s),Y(s))ds,
\]
is a continuous martingale for every $f\in\mathcal{D}(\bar{\mathcal{L}})$, it is not hard to show that it is a local martingale for all $f\in\hat{\mathcal{D}}(\mathcal{L})$.

First, we claim that the process $Y(t) - \int_0^t B(X(s),Y(s))ds$ does not move unless $X(t)=0$. To see this, let $|x|>\delta>0$ and $\tau = \inf\{t>0: |X(t)|\leq \delta\}$. Then by the martingale problem with a smooth function satisfying \
\[
f(x,y)=\left\{\begin{array}{cc}
y_k& |x|>\delta\\
0&|x|<\delta/2
\end{array}\right. ,
\]
we have that
\[Y_k(\tau \wedge t) - \int_0^{\tau \wedge t} B_k(X(s),Y(s))ds\]
is a local martingale. Applying the same procedure with $y_k^2$ shows that this local martingale has quadratic variation $0$ and therefore does not move. The boundary conditions are not relevant for this stopped process because we stop the process before it reaches the boundary.
To study the motion at the boundary, we note that for any $k\in\{1,\dots,d\}$, the function
\[f(x,y) = y_k - |x| \beta_k(y)\]
satisfies the boundary conditions \eqref{eq:appendix-bdry-conds}. Because $(X,Y)$ solves the martingale problem,
\begin{align} \label{eq:Y-martingale}
  Y_k(t) - |X(t)| \beta_k(Y(t)) &- y_k - |x| \beta_k(y)- \int_0^t B_k(X(s),Y(s)) ds -\\
\nonumber &- \int_0^t \sum_{j=1}^d |X(s)| \frac{\partial \beta_k}{\partial y_j}(Y(s))B_j(X(s),Y(s))ds
\end{align}
is a local martingale. Note that  because $f(x,y)=x$ satisfies the boundary conditions, we know that $X(t)$ is also a local martingale.  Therefore by the Meyer-Tanaka formula,
\begin{equation*}
 |X(t)| = |x| + L^X(t,0) + \textnormal{ local martingale}.
\end{equation*}
Next, we observe that for any partition $0=t_0<t_1<..<t_{N+1} = t$,
\begin{align*}
  |X(t)|\beta_k(Y(t)) &- |x|\beta_k(y) =\\
&=  \sum_{i=0}^N \beta_k(Y(t_{i})) (|X(t_{i+1})| - |X(t_i)|) + \sum_{i=0}^N |X(t_{i+1})| \left(\beta_k(Y(t_{i+1})) - \beta_k(Y(t_i)) \right).
\end{align*}
As the partition gets finer,
\[\lim_{\|P\| \to 0} \sum_{i=0}^N \beta_k(Y(t_i)) (|X(t_{i+1})| - |X(t_i)|) = \int_0^t \beta_k(Y(s)) L^X(ds,0) + \textnormal{ local martingale}. \]
in $L^2(\Omega)$. As for the second sum,
\begin{align*}
 &\sum_{i=0}^N |X(t_{i+1})| \left(\beta_k(Y(t_{i+1})) - \beta_k(Y(t_i)) \right)\\
  &= \sum_{i=0}^N |X(t_{i+1})| \left( \beta_k\left( \tilde{Y}(t_{i+1}) + \int_0^{t_{i+1}} B(X(s),Y(s))ds \right) - \beta_k \left(\tilde{Y}(t_i) + \int_0^{t_i} B(X(s),Y(s)) ds \right) \right)
\end{align*}
where
\[\tilde{Y}(t) = Y(t) - \int_0^t B(Y(s)) ds,\]
which does not move unless $X(t) =0$ with probability one. Given any partition, $P$ of $[0,t]$, we can refine the partition (randomly) in the following manner. For any $I=(t_{i},t_{i+1}) \in P$,
\begin{enumerate}
  \item If there exists $s \in I$ such that $X(s)=0$ but  $X(t_{i+1})\neq 0$, we split the interval, $I$, at the random time
  \[\tau = \sup\{s \in (t_i, t_{i+1}): X(s)=0\}\]
  \item Otherwise, we do not refine $I$. 
\end{enumerate}
Call the randomly refined partition $\tilde{P}$. This partition has the property that if $[t_i,t_{i+1}] \in \tilde{P}$,
then either $X(t_{i+1})=0$ or $\tilde{Y}(t_{i+1}) = \tilde{Y}(t_i)$. Switching to this random family of partitions,
\begin{align*}
&\lim_{\|\tilde{P}\| \to 0} \sum_{k=0}^N |X(t_{i+1})| \left( \beta_k\left( \tilde{Y}(t_i+1) + \int_0^{t_{i+1}} B(Y(s))ds \right) - \beta_k \left(\tilde{Y}(t_i) + \int_0^{t_i} B(Y(s)) ds \right) \right)\\
&= \int_0^t \sum_{j=1}^d |X(s)| \frac{\partial \beta_k}{\partial y_j}(Y(s)) B_j(Y(s)) ds
\end{align*}
in $L^2(\Omega)$. In this way, we have proven that
\begin{equation} \label{eq:|X|beta(Y)}
  |X(t)|\beta_k(Y(t)) - |x|\beta_k(y) \stackrel{a.s}{=} \int_0^t \beta_k(Y(s))L^X(ds,0) + \int_0^t \sum_{j=1}^d |X(s)| \frac{\partial \beta_k}{\partial y_j}(Y(s)) B_j(X(s),Y(s)) ds
\end{equation}
and by \eqref{eq:Y-martingale},
\begin{equation*}
  M(t) := Y(t) - y - \int_0^t B(X(s),Y(s)) ds - \int_0^t \beta(Y(s))L^X(ds,0)
\end{equation*}
is a local martingale.
 By the Ito formula,
for any $f \in C^2(\reals^{d+1})$, 
\begin{align} \label{eq:appendix-ito-fmla}
  &f(X(t),Y(t)) = f(x,y) + \sum_{k=1}^d \int_0^t f_{y_k}(X(s),Y(s)) B_k(X(s),Y(s)) ds \nonumber\\
  &+  \sum_{k=1}^d \int_0^t f_{y_k}(X(s),Y(s)) \beta_k(Y(s)) L^X(ds,0) + \frac{1}{2} \int_0^t f_{xx}(X(s),Y(s)) d\left< X\right>_s\nonumber \\
   &+ \frac{1}{2} \sum_{i,j=1}^d \int_0^t f_{y_iy_j}(X(s),Y(s)) d\left< M_i,M_j\right>_s +  \sum_{j=1}^d \int_0^t f_{xy}(X(s),Y(s)) d \left< X, M_j \right>_s\nonumber \\
   &+ \textnormal{local martingale}.
\end{align}
We compare this to the martingale problem.
Let  $f \in C^3(\mathbb{R}^{d+1})$ and define
\[g(x,y) = f(x,y) - \sum_{k=1}^d |x|\beta_k(y) f_{y_k}(0,y) - \frac{1}{2} \sum_{i,j=1}^d |x| \alpha_{ij}(y) f_{y_i y_j}(0,y)\]
 This $g$ satisfies the boundary conditions \eqref{eq:appendix-bdry-conds} and therefore
\[g(X(t), Y(t)) - \int_0^t \mathcal{L}g(X(s),Y(s))ds
\]
is a martingale.
That is,
\begin{align} \label{eq:appendix-martingale-prob}
  &f(X(t),Y(t)) = f(x,y) + \sum_{k=1}^d \left( |X(t)| \beta_k(Y(t)) f_{y_k}(0, Y(t)) -|x|\beta_k(y) f_{y_k}(0,y) \right)\nonumber \\
  &- \sum_{k,j=1}^d \int_0^t |X(s)| \frac{\partial}{\partial y_j} \left(\beta_k(Y(s)) f_{y_k}(0,Y(s)) \right) B_j(X(s),Y(s)) ds \nonumber \\
  &+ \sum_{i,j=1}^d \left( |X(t)| \alpha_{ij}(Y(t)) f_{y_i y_j}(0,Y(t)) -  |x|\alpha_{ij}(y) f_{y_i y_j}(0,y) \right) \nonumber \\
  & - \sum_{i,j,k=1}^d \int_0^t |X(s)| \frac{\partial}{\partial y_k} \left(\alpha_{ij}(Y(s)) f_{y_i y_j}(X(s),Y(s)) \right) B_k(X(s),Y(s)) ds   \\
  &+ \sum_{k=1}^d \int_0^t f_{y_k}(X(s),Y(s)) B_k(X(s),Y(s))ds  + \frac{1}{2} \sum_{i,j=1}^d \int_0^t A(X(s),Y(s)) f_{xx}(X(s),Y(s)) ds  \nonumber  \\
  &+ \textnormal{local martingale}.
\end{align}
By \eqref{eq:|X|beta(Y)}, it follows that
\begin{align*}
 &|X(t)| \beta_k(Y(t)) f_{y_k}(0, Y(t)) -|x|\beta_k(y) f_{y_k}(0,y)\\
 &  - \sum_{j=1}^d \int_0^t |X(s)| \frac{\partial}{\partial y_j} \left(\beta_k(Y(s)) f_{y_k}(0,Y(s)) \right) B_j(X(s),Y(s)) ds \\
 & = \int_0^t \beta_k(Y(s)) f_{y_k}(X(s),Y(s)) L^X(ds,0)
\end{align*}
and
\begin{align*}
  &|X(t)| \alpha_{ij}(Y(t)) f_{y_i y_j}(0,Y(t)) - |x| \alpha_{ij}(y) f_{y_i y_j}(0,y) \\
  &- \sum_{k=1}^d \int_0^t |X(s)| \frac{\partial}{\partial y_k} \left( \alpha_{ij}(Y(s)) f_{y_i y_j}(X(s),Y(s)) \right) B_k(X(s),Y(s)) ds       \\
  &= \int_0^t \alpha_{ij}(X(s),Y(s)) f_{y_i y_j}(X(s),Y(s)) L^X(ds,0).
\end{align*}
Then by comparing \eqref{eq:appendix-ito-fmla} and \eqref{eq:appendix-martingale-prob} we can conclude that

\begin{align*}
  &\left<X,M_j \right>_t = 0,\\
  &\left<X \right>_t = \int_0^t A(X(s),Y(s))ds, \\
  &\left< M_i, M_j \right>_t = \int_0^t \alpha_{ij}(Y(s))L^X(ds,0).
\end{align*}
Finally, we define
\begin{align*}
  &W(t) := \int_0^t A^{-1/2}(X(s),Y(s)) dX(s),
  &V^X(t) := \int_0^t \alpha^{-1/2}(X(s),Y(s)) dM(s).
\end{align*}
By Levy's theorem, $W(t)$ is a one-dimensional Brownian-motion. Similarly, if
\[L^{-1}(s) = \inf\{ t>0: L^X(t,0)=s\},\]
then by Levy's theorem
\begin{equation*}
  W_0(s) := V^X(L^{-1}(s))
\end{equation*}
is a $d$-dimensional Wiener process,
and $V^X(t) = W_0(L^X(t,0))$.
Finally, we can conclude that  $(X(t),Y(t))$ is the weak solution of
\begin{equation}
  \begin{cases}
    dX(t) = \sqrt{A(X(t),Y(t))} dW(t),\\
    dY(t) = B(X(t),Y(t))dt + \beta(Y(t))L^X(dt,0) + \sqrt{\alpha(Y(t))}dV^X(t).
  \end{cases}
\end{equation}

\end{proof}

\begin{Remark}
In the case when $p_+(y)$ and $p_-(y)$ are not necessarily equal, one direction is still given by Lemma \ref{lem:itm}. The main difference in proving the other direction is that now $X(t)$ is not a local martingale. Rather, by applying the martingale problem to the function $f(x,y)=x-(p_+(y)-p_-(y))|x|$ (which requires sufficient regularity from $p_{\pm}$), we see that
\begin{equation}\label{eq:Xeqnonsym}
X(t)-(p_+(Y(t))-p_-(Y(t))|X(t)|
\end{equation}
is a local martingale. The difficulty lies in the fact that we do not a priori know that $X(t)$ is a semimartingale and therefore we cannot immediately apply the Tanaka-Meyer formula. However, let $\tau_0^{\delta}=0$, $\sigma_0^{\delta}=\inf\{t\geq 0| X(t)=0\}$ and recursively define
\[
\tau_i^{\delta}=\inf\{t>\sigma_{i-1}^{\delta}| |X(t)|=\delta\},\qquad\sigma_{i}^{\delta}=\inf\{t>\tau_i^{\delta}|X(t)=0\}
\]
and we can write
\[
X(t)-X(0)=\sum_{i=0}^{\infty}\left[X(t\wedge\sigma_i^{\delta})-X(t\wedge\tau_i^{\delta})\right]+\sum_{i=1}^{\infty}\left[X(t\wedge\tau_i^{\delta})-X(t\wedge\sigma_{i-1}^{\delta})\right].
\]
It can be shown similarly as in Lemma 2.2 in \cite{FS99} that as $\delta\downarrow 0$, the first sum converges to a local martingale while the second sum converges to a monotone process that only changes when $X(t)=0$. More precisely, this second term converges to $\int_0^t(p_+(Y(s))-p_-(Y(s)))dL^X(s)$. This shows that $X$ is a semimartingale and (\ref{eq:Xeqnonsym}) implies that $X(t)-\int_0^t(p_+(Y(s))-p_-(Y(s)))dL^X(s)$ is a local martingale. The rest of the proof goes through with minor modifications.
\end{Remark}
\end{appendix}

\bibliographystyle{amsplain}
\bibliography{citations}

\bibliographystyle{amsplain}

\end{document}